\documentclass[a4paper,12pt]{article}
\usepackage{a4wide}
\usepackage{amsmath}
\usepackage{amsfonts, amssymb, amsthm}
\usepackage[usenames]{color}
\usepackage{mathrsfs}
\usepackage{verbatim}
\usepackage{enumerate}

\usepackage{graphicx} % Pour inclure des graphiques 
\usepackage{caption}

\newcommand{\R}{\mathbb R}
\newcommand{\dv}{{\rm div \;}}
\newcommand{\Hh}{\mathcal{H}}
\newcommand{\M}{\mathcal M}
\newcommand{\N}{\mathbb N}
\newcommand{\Ss}{{\mathbb S}^1}
\newcommand{\ve}{\varepsilon}

\newcommand{\deb}{\rightharpoonup}
\newcommand{\destar}{\overset{*}\deb}

\newcommand{\modvec}{|}

\newcommand{\eps}{\varepsilon}

\newcommand{\Om}{\Omega}

\newcommand{\dphi}{d_{\varphi}}

\newcommand{\Dv}{{\rm Div \;}}
\newcommand{\Ghlambeps}{G^h_{\Lambda,\eps}}

\usepackage[latin1]{inputenc}
 \usepackage[usenames,dvipsnames]{pstricks}
\usepackage{epsfig}
\usepackage{pst-grad} % For gradients
\usepackage{pst-plot} % For axes

\newtheorem{thm}{Theorem}[section]
\newtheorem{prop}[thm]{Proposition}

\newtheorem{lem}[thm]{Lemma}

\newtheorem{defn}[thm]{Definition}

\newtheorem{preremark}[thm]{Remark}
\newenvironment{remark}{\begin{preremark}\rm}{\medskip \end{preremark}}

\newcommand{\uspace}{W^{1,p}_K(\Omega)}

\DeclareMathOperator{\spt}{spt}
\DeclareMathOperator{\diam}{diam}

\DeclareMathOperator{\lip}{Lip}

\author{Matthieu Bonnivard\footnotemark[1], Antoine Lemenant\footnotemark[1], Filippo Santambrogio\footnotemark[2]\footnotemark[2]}

\title{Approximation of length minimization problems\\ among compact connected sets.}

\begin{document}

\footnotetext[1]{Université Paris-Diderot, Laboratoire Jacques-Louis Lions, CNRS UMR 7598, bonnivard@ljll.univ-paris-diderot.fr, lemenant@ljll.univ-paris-diderot.fr}
\footnotetext[2]{Université Paris-Sud, Laboratoire de Math\'ematiques d'Orsay, CNRS UMR 8628, filippo.santambrogio@math.u-psud.fr}
\maketitle

\abstract{In this paper we provide an approximation \emph{à la} Ambrosio-Tortorelli of some classical minimization problems involving the length of an unknown one-dimensional set,  with an additional connectedness constraint, in dimension two. We introduce a term of new type relying on a weighted geodesic distance that  forces the minimizers to be connected at the limit. We apply this approach to approximate the so-called  Steiner Problem, but also the average distance problem, and finally a problem relying on the $p$-compliance energy. The proof of convergence of the approximating functional, which is stated in terms of $\Gamma$-convergence  relies on technical tools from geometric measure theory, as for instance a uniform lower bound for a sort of average directional Minkowski content of a family of compact connected sets. }

%\tableofcontents

\section{Introduction}

In the pioneering work by Modica and Mortola \cite{MM} (see also \cite{alberti}), it is proved that the functional 
$$M_\varepsilon(\varphi)=\varepsilon \int_{\Omega}\modvec \nabla \varphi \modvec ^2dx+\frac{1}{\varepsilon}\int_{\Omega}\varphi^2 (1-\varphi)^2 dx$$
converges as $\varepsilon\to 0$ in a sense to be precised, to the Perimeter functional 
$$P(\varphi)=\frac{1}{3}Per(A,\Omega) \text{ if } \varphi={\bf 1}_{A} \in BV(\Omega), \text{ and }+\infty \text{ otherwise}.$$
The original motivation for Modica and Mortola was a mathematical justification of convergence for some two phase problem based upon a model by Cahn and Hilliard.

Later, this procedure gave rise to a method to perform a numerical approximation of a wide class of variational problems. Indeed, in order to minimize numerically the geometrical functional given by the perimeter, one may find it convenient to minimize the more regular functional $M_\varepsilon$, which are of elliptic type. This idea was used by many authors in the last two decades, with quite satisfactory results. Let us mention for instance Ambrosio-Tortorelli \cite{at} for the Mumford-Shah functional,  and more recently Oudet \cite{OudetKelvin} and Bourdin-Bucur-Oudet \cite{bbo} for partition problems, Santambrogio and Oudet-Santambrogio \cite{sant,OudetSantam} for branched transport,  among many others like \cite{alr,dmi} for fractures, etc.

But no analogous method of approximation was given so far to approximate one of the simplest problem of that kind, namely the Steiner problem. Given a finite number of points $D:=\{x_i\}\subset \R^2$, the original Steiner problem consists in minimizing
\begin{eqnarray}
\min\big \{\Hh^1(K) \quad ; \; K\subset \R^2 \text{ compact, connected, and containing } D \big\}. \label{Steiner00}
\end{eqnarray}
Here, $\Hh^1(K)$ stands for the one-dimensional Hausdorff measure of $K$. It is  known that minimizers for \eqref{Steiner00} do exist, need not to be unique, and are  trees composed by a finite number of segments joining with only triple junctions at 120°, whereas  computing a minimizer is very hard (some versions of the Steiner Problem belong to the original list of (NP)-complete problems by Karp, \cite{Karp}). We refer for instance to \cite{gilbpoll} for a history of the problem and to \cite{ps2009} for recent mathematical results about it.

The main difficulty is to take care of the connectedness constraint on the competitors. In \cite{OudetSantam}, the problem of branched transport is considered which is slightly related to the Steiner problem. We mention that a tentative approach to the Stenier problem, by considering it as a limit case $\alpha=0$ in the branched transport problem, was indeed performed in \cite{OudetSantam}, but the analysis in such a paper required $\alpha>0$ and the simulation performed with $\alpha$ close to $0$ were not completely satisfactory.

In this paper we propose a new way to tackle the connectedness constraint on minimizers. Our strategy is to  add a term  in the Modica-Mortola functional, relying on the weighted geodesic distance  $d_\varphi$, defined as follows. Let $\Omega \subset \R^2$. For any $\varphi \in  C^0(\overline{\Omega})$ we define the geodesic distance weighted by $\varphi$ as being
\begin{eqnarray}
d_\varphi(x,y):= \inf \left\{ \int_{\gamma} \varphi(x) d\mathcal{H}^1(x) ; \; \gamma \text{ curve connecting } x \text{ and } y \right\}. \label{defDphi}
\end{eqnarray}
Given a function $\varphi$, the distance $d_\varphi$ can be treated numerically by using the so-called ``fast marching'' method (see \cite{sethian-fastmarching}). A recent improvement of this algorithm  (see \cite{bcps}) is now able to compute at the same time $d_\varphi$ and its gradient with respect to $\varphi$, which is especially useful when dealing with a minimization problem on $\varphi$ involving the values of $d_\varphi$. Our proposal to approximate the problem in \eqref{Steiner00} is to minimize the functional
\begin{eqnarray}
\frac{1}{4\varepsilon}\int_{\Omega}(1-\varphi)^2dx + \varepsilon\int_{\Omega}\modvec \nabla \varphi\modvec ^2dx+\frac{1}{c_\varepsilon}\sum_{i = 1}^{N}d_\varphi(x_i,x_1),\label{approxfunctSt} 
\end{eqnarray}
with $c_\varepsilon\to 0$, arbitrarily. The first two terms are a simple variant of the standard Modica-Mortola functional, already used in  \cite{at}: they force $\varphi$ to tend to $1$ almost everywhere  at the limit $\ve\to 0$, and the possible transition to a thin region where $\varphi=0$ is paid by means of the length of this region, while the last term tends to make $\varphi$ vanish on a connected set containing the points $x_i$. The key point is that whenever
$$\sum_{i = 1}^{N}d_\varphi(x_i,x_1)=0,$$
 the set $\{\varphi =0\}$ must be path connected and contain the $\{x_i\}$. In our general setting, we will rather consider quantities of the form 
 $$\int_{\Omega}d_{\varphi}(x,x_0)d \mu,$$
 leading to a more general Steiner problem where the finite set $\{x_i\}$ is replaced by the support of $\mu$. We refer the reader to Section \ref{sectionSteiner} and Theorem \ref{SteinerTheorem} for a precise statement about the Steiner Problem.

Beyond the Steiner problem, we also consider in this paper  two other minimization problems that one can find in the literature, where the admissible competitors are again compact connected sets. The first one is the average distance problem studied in \cite{bos,bs,s,ps,bos,st,bms,l1,t,ant,lusl,sl}. For some $\Omega\subset \R^2$,  $\Lambda >0$,  and  $f\geq 0$ a positive $L^1$ density  on $\Omega$, the average distance problem  is the the following:
\begin{eqnarray}
\min\left\{\int_{\Omega} dist(x,K)  f(x)dx +\Lambda \Hh^1(K)\; ; \; K \subset \overline{\Omega} \text{ is closed and connected}\right\}. \label{prob1}
\end{eqnarray}
Here, $dist$ refers to the euclidean distance in $\R^2$. Notice that the first version of this problem which has been proposed in \cite{bos} replaced the length penalization $+\Lambda \Hh^1(K)$ with a constraint $\Hh^1(K)\le L$. Mathematically the two problems are similar, but the penalized one is easier to handle, and we concentrate on it in our paper.

The third and last problem  is a minimization problem involving the so-called $p$-Compliance energy coming from classical mechanics, see for instance    \cite{Buttst} for the following  version of that  problem: let $p \in ]1,+\infty[$ and $f\in L^{(p^*)'}(\Omega)$, where $p^*$ is the exponent of the injection $W^{1,p}\in L^{p^*}$ and $(p^*)'$ is its dual. For any closed set $K \subset \overline{\Omega}$, we denote by $W^{1,p}_{K}(\Omega)$ the adherence in $W^{1,p}(\Omega)$ of functions $C^\infty_c(\R^2\setminus K)$  (i.e. we put Dirichlet conditions on $K$ but not on $\partial\Omega$). We then consider $u_K$, the  unique solution for the problem
\begin{equation}
\min \left\{ \int_{\Omega}\Big(\frac{1}{p}\modvec \nabla u\modvec ^p -uf\Big)\; dx \;;\; u \in W^{1,p}_{K}(\Omega) \right\}. \label{prob1bis}
\end{equation}
In particular, denoting $\Delta_pu=\dv (\modvec \nabla u\modvec ^{p-2}\nabla u)$ the standard $p$-Laplacian, $u_K$ is a weak solution for the problem
$$
\left\{
\begin{array}{cc}
-\Delta_p u = f & \text{ in  } \Omega \setminus K \notag \\
u = 0 & \text{ on } K\notag\\
\frac{\partial u}{\partial n} =0  & \text{ on } \partial \Omega\setminus K \notag 
\end{array}
\right.
$$
The $p$-Compliance problem is then defined as being
\begin{equation}
\min \left\{ \frac{p-1}{p}\int_{\Omega}u_K f\;dx +\Lambda \Hh^1(K)\;;\; K \subset \overline{\Omega} \text{ is closed and connected} \right\}, \label{pb20}
\end{equation}
which is a min-max problem  (minimization in $K$ of a max in $u$).

In \cite{Buttst} it is observed that the above two problems are intimately related.  Precisely, when $p$ converges to $+\infty$, the $p$-compliance functional $\Gamma$-converges to the average distance functional (for instance if one decided to look at the variant where the function $u_K$ in the $p-$compliance is defined with $u=0$ on $K\cup\partial\Omega$, then at the limit one would get the average distance functional $\int_\Omega dist(x,K\cup\partial\Omega) f(x)dx$).

In order to approximate the two problems we first use a duality argument to turn them  into the following form. We denote by
$$\mathcal{K}:= \{K \subset \overline{\Omega}; \text{ closed and connected}\},$$
$$\mathcal{A}_p:=\Big \{(K,v)\,:\,K\in\mathcal{K},\,v \in L^q(\Omega) ; -\dv v  = f \text{ in }\overline\Omega \setminus K  \text{ and  } v\cdot n_\Omega=0 \text{ on }\partial \Omega \setminus K\Big\},$$
where $q$ is the conjugate exponent to $p$. Notice that the condition 
$$-\dv v  = f \text{ in }\overline\Omega \setminus K \text{ and  } v\cdot n_\Omega=0 $$ 
is intended as $\int_\Omega v\cdot \nabla\phi=\int f\phi$ for every $\phi\in C^1(\overline\Omega)$ with $\phi=0$ on $K$.
By duality, Problem \eqref{pb20} is equivalent to 
\begin{equation}
 \inf_{(K ,v) \in \mathcal{A}_p} \left\{ \frac{1}{q}\int_{\Omega}\modvec  v\modvec ^q dx + \Lambda \Hh^1(K)\right\} , \label{pb30}
\end{equation}
and Problem \eqref{prob1} is equivalent to \eqref{pb30} with $q=1$. More precisely, the $L^q$ norm is replaced by the norm in the space of measures and the admissible set by
%{\color{red} j'ai mis partout des conditions de Neumann}
$$\mathcal{A}_\infty:=\Big\{(K,v)\,:\,K\in\mathcal{K},\;v \in \mathcal M^d(\overline\Omega) ; -\dv v  = f \text{ in }\Omega \setminus K  \text{ and  } v\cdot n_\Omega=0 \text{ on }\partial \Omega \setminus K\Big\}.$$
In the definition of $\mathcal{A}_p$ and $\mathcal{A}_\infty$, the condition on the divergence and the normal derivative are intended in a weak sense. We refer to Section 5 for the precise proofs of these facts.

The purpose of introducing this new formulation is that we obtain a min-min problem, which is more easy the handle than the original one, of min-max type.

To approximate the problem \eqref{pb30} we then consider the family of functionals 
$$F_{\varepsilon}(v,\varphi,y)=
\frac{1}{q}\int_{\Omega} \modvec v\modvec ^q dx + \frac{1}{4\varepsilon}\int_{\Omega}(1-\varphi)^2 dx + \varepsilon\int_{U}\modvec \nabla \varphi \modvec ^2 dx $$
$$\hspace{1cm} +\frac{1}{\sqrt{\varepsilon}}\int_{\Omega}d_{\varphi}(x,y) d\big|\dv v +f\big|(x) +  |\dv v|(\Omega)
$$
which is proved to $\Gamma$-converge, as $\varepsilon \to 0$, to a functional whose minimization is equivalent to that of \eqref{pb30}, and this convergence is  one of  our main results (Theorem \ref{Gammaconv}). Notice that the framework for the Steiner problem, where no vector field $v$ is involved, is a bit different. We do not obtain a true $\Gamma-$convergence result, but we are anyway able to prove the convergence to a minimizer.

Each term of $F_\varepsilon$ is easy to interpret: the first term is just the original functional on $v$, the second and third are the Modica-Mortola part, the fourth term is the one which forces $\varphi$ to vanish on an almost connected set containing the support of $|\dv v +f|$, and finally the last term is just here to guarantee compactness in the space of measures. This term does not dissapear at the limit but does not affect the minimization as we prove in Proposition \ref{propositionFilippo}.

Let us further say a few words about the coefficient $1/\sqrt{\varepsilon}$ in front of the fourth term.  Indeed, in \eqref{approxfunctSt} $\sqrt{\ve}$ was replaced by a more general coefficient $c_\ve\to 0$. This was enough for Steiner problem, while here the assumption has to be refined because of some boundary issues. Indeed, in the proof of the $\Gamma$-limsup inequality, we follow a standard construction to find a recovering sequence of functions  $\varphi_\varepsilon$ which are almost zero in a neighborhood of a given  candidate set $K$, and we want them to be equal to $1$ on $\partial\Omega$. In order to do that, we need to modify them a little bit and in particular to contract them inside  the domain, if $K$  touches the boundary (see Lemma \ref{limsupmodifie}). This gives an extra cost, of the order of $(\varepsilon|\ln(\varepsilon)|)$. Hence, in our previous formula, any coefficient of the form $o((\varepsilon|\ln(\varepsilon)|)^{-1})$ could work, but we chose $\ve^{-1/2}$ for simplicity.  On the contrary, we decided to avoid this kind of boundary issues for the Steiner problem, since we know that the optimal set lies in the convex hull of the points $x_i$, that we can assume compactly contained in $\Omega$.

The proof of $\Gamma$-convergence uses standard ingredients from Geometric measure theory but relies on original ideas as well. If the $\Gamma-\limsup$ is based on classical arguments,  at the contrary the $\Gamma-\liminf$ is more involved and represents the biggest part of the proof. For instance, no slicing argument is permitted due to the weak regularity of $v$ (the divergence of $v$ is a measure, but this does not mean $v\in BV$ in general). Thus the liminf inequality has to be performed directly in dimension 2 via new techniques.  Also, some objects related to our new term with $d_\varphi$ have to be introduced: in the limiting process,  the good set to consider is rather $\{d_\varphi(x,x_0)\leq \delta \}$  than the usual $\{\varphi(x)\leq \delta\}$ which  leads to some technical issues.

\paragraph{Structure of the paper} Section 2 presents several technical tools from different origins: we start with well-known notions on the measure $\Hh^1$ and on connected sets that we want to recall, then we introduce and study the geometrical quantity $I_\lambda$ which will be crucial in our analysis, we state an easy estimate on the total variation of functions of one variable, and finally we produce a recovery sequence useful in the $\Gamma-\limsup$ inequality following a standard construction and adapting it to our needs.
 Later, Section 3 gives the proof of a $\Gamma-\liminf$-type inequality that we need for all the three problems. In section 4 we explain how to handle the Steiner problem and in Section \ref{Section-approx-average} both the average distance and the $p-$compliance problem.  At the end of Section \ref{Section-approx-average}, a discussion about the existence of minimizers for $\ve>0$ in our approximation is given.  Finally, in the the last Section 6 we use the approximation result that we proved in order to produce some numerical results. For the sake of simplicity, only the $2-$compliance problem has been considered, but we are confident that adaptations and improvements of the same method could be adapted (and improved) to more general settings.

\paragraph{Context and Acknowledgments} The work of the authors is both part of the work of the ANR research project GEOMETRYA, ANR-12-BS01-0014-01, and of the PGMO research project MACRO ``Mod\`eles d'Approximation Continue de R\'eseaux Optimaux''. The support of these research projects is gratefully acknowledged. The second and third author announced the theoretical part of this work, with particular attention to the case of the Steiner problem in \cite{LemSanCRAS}. The question of applying these techniques to the Steiner problem was raised at the Italian Conference on Calculus of Variations in Levico Terme, February 2013, by G. P. Leonardi and E. Paolini. We sincerely acknowledge them for the fruitful discussions, as well as A. Chambolle for some initial suggestions about the dual formulation of the Compliance problem, which lead to the genesis of this paper. The general question of finding an energy which forces connectedness of minimizers was first asked by Edouard Oudet to the second author in Pisa in the early 2010. We acknowledge him for bringing out to our attention this nice question.
 
%%%%%%%%%%%%%%%%%%%%%%
\section{Technical tools}

Our proof of $\Gamma$-convergence will use  some tools from the Geometric measure theory. For the sake of clarity, we present   in this section some particular results of independent interest that will be used later. Some of them are standard but difficult to find in the literature (like for instance Proposition \ref{rectifiability} about existence of tangent line in connected sets of finite length, or the covering Lemma \ref{covering1}), but some others are more original  like Lemma \ref{Lemme2} which is one of the key estimate leading to our main result.

%\subsection{Review on weak convergence of measures}

%In order to make the paper self contained, we mention in this short paragraph, some classical results on the weak convergence of measure that will be used later. We refer the reader to \cite[Section 1.4.]{afp} for the proofs.

%Let $X\subset \R^2$ be a compact set. We denote by $\mathcal{M}(X)$ the finite Borel measures supported on $X$. If   $\mu \in \mathcal{M}(X)$ we denote by $|\mu|(X)$ its total variation. If $\mu_\varepsilon$ is a family of measures in $\mathcal{M}(X)$, we say that $\mu_\varepsilon$ weakly$-*$ converges to $\mu \in \mathcal{M}(X)$ whenever 
%$$\int_{X}f d\mu_\varepsilon \to \int_{X}f d\mu,$$
%for all continuous function $f\in C^0_c(X)$.

%\begin{thm} {\rm \cite[Theorem 1.59]{afp}} Let $X\subset \R^2$ be a compact set. If $\mu_\varepsilon$ is a sequence of measures in $\mathcal{M}(X)$ such that $\sup_\varepsilon\{|\mu_\varepsilon|(X)\}<+\infty$ then it has a weakly$-*$ converging subsequence. Moreover $\mu\mapsto |\mu|(X)$ is weakly$-*$ lower semicontinuous with respect to weakly$-*$ convergence.
%\end{thm}

%\subsection{Some classical tools about the Hausdorff measure}

%%%%%%%%%%%%%%%%%%%%%

Let $A$ be a subset of $\R^2$. The Hausdorff measure of $A$ is
\begin{eqnarray}\label{defHH}
\Hh^1(A)=\lim_{\tau \to 0^+} \Hh^1_{\tau}(A),
\end{eqnarray}
where
\begin{eqnarray}
 \Hh^1_{\tau}(A):=\inf \left\{ \sum_{i=1}^{+\infty} \diam(A_i) \quad ; \quad A\subset
\bigcup_{i=1}^{+\infty}A_i \; \text{ and } \; \diam(A_i)\leq \tau\right\} .\label{infinmum}
\end{eqnarray}
It is well known that $\mathcal{H}^1$ is an outer measure on $\R^2$ for which the Borel sets are measurable sets. Moreover, its restriction to lines coincides with the standard Lebesgue measure. We will also denote by $d_H$ the Hausdorff distance between two closed sets $A$ and $B$ of $\R^2$ defined by
$$d_H(A,B):=\max\big \{\sup_{x\in A} dist(x,B),\sup_{x\in B}dist(x,A)\big \}.$$
%{\color{red} somme ou max entre les deux ?}
First we recall a standard result about densities.

\begin{prop}{\rm \cite[Theorem 2.56.]{afp} } \label{densityLemma} Let $\Omega\subset \R^2$,  $\mu$ be a positive Radon measure in $\Omega$ and let $A\subset \Omega$ be a Borel set such that
$$\limsup_{r\to 0}\frac{\mu(B(x,r))}{2r}\geq 1 \quad \forall x\in A.$$
Then
$$\mu\geq \Hh^1|_{A}.$$
\end{prop}

%%%%%%%%%%%%%%%%%%%%%%%%%%%

It is well known that compact connected sets with finite one dimensional Hausdorff measure enjoy some nice regularity  properties that we summarize in the following.

%We recall that $K\subset \R^2$ is one-rectifiable if it is a Borel set such that $\Hh^1|_{K}$ is sigma-finite and there exist a countable family $\{\Gamma_m\}$ of $C^1$ curves in $\R^2$ and a Borel set $E$ such that $\Hh^1(E)=0$ and
%$$K\subset E\cup \left(\bigcup_{m} \Gamma_m\right).$$

\begin{prop} \label{rectifiability}Let $K\subset \R^2$ be a compact connected set such that $L:=\Hh^1(K)<+\infty$. Then  :
\begin{enumerate}[{\rm($i$)}]
\item There exists a Lipschitz surjective mapping $f:[0,L]\to K$. In particular $K$ is arcwise connected.
\item For $\Hh^1$-a.e. $x \in K$ there exists a tangent line  $T_x\subset \R^2$, in the sense that
$$\lim_{r\to 0} \frac{d_H\big(K\cap B(x,r), T_x \cap B(x,r)\big)}{r} = 0.$$
\item For all $ \varepsilon \in (0,1/2)$ there exists some $r_0>0$ such that
\begin{eqnarray}
\pi(K\cap B(x, r)) \supseteq  T_x\cap B(x,(1-\varepsilon)r)\quad \forall r\leq r_0, \label{measproj}
\end{eqnarray}
where $\pi :\R^2\to T_x$ denotes the orthogonal projection onto the line $T_x$ identified with $\R$ with origin at  $x$.
\end{enumerate}
\end{prop}
\begin{proof} The proof of $(i)$ can be found for instance in \cite[Proposition 30.1. page 186]{d}. We now use $(i)$ to prove $(ii)$. Indeed, a direct consequence of $(i)$ is that $K$ is rectifiable, which implies that $K$ admits an approximate tangent line at  $\Hh^1$-a.e. point $x \in K$ (see for e.g. \cite[Theorem 15.19]{mat}).  This means that there exists a line $T_x \subset \R^2$ for which
\begin{eqnarray}
\lim_{r\to 0}\frac{\Hh^1\big(K\cap \big(B(x,r)\setminus W(x,s)\big)\big)}{r}=0 \quad \forall s>0, \label{density10}
\end{eqnarray}
where $W(x,s)$ is the open cone of aperture $s$ defined by
$$W(x,s):=\Big\{y\in \R^2; d(y,T_x)< s\modvec y-x\modvec \Big\}.$$
We will prove that, from the connectedness of $K$, the above measure-theoretical estimate may be strengthened. In particular, we claim that
\begin{eqnarray}
\lim_{r\to 0}  \frac{\sup\{d(y,T_x) ; y \in K\cap B(x,r)\}}{r}=0.\label{firsthalf}
\end{eqnarray}
This follows from the fact that due to the connectedness of $K$, we automatically have a lower Ahlfors-regular condition, namely
$$\Hh^1(K\cap B(x,r))\geq r \quad \forall x \in K, \; \forall r<\diam(K)/2.$$
Consequently, for $r$ small enough (depending on $s$) the set
$$K\cap \big(B(x,2r)\setminus B(x,r)\big)\setminus W(x,2s)$$ is empty, otherwise $K\cap B(x,2r)\setminus W(x,s)$ would contain a piece of measure at least $sr/10$ which would be a contradiction to \eqref{density10} (this is actually \cite[Exercice 41.21.3. page 277]{d}). If the sets $K\cap \big(B(x,2r)\setminus B(x,r)\big)\setminus W(x,2s)$ are empty for every small $r$, by union the same will be true for $K\cap \big(B(x,2r)\setminus \{x\}\big)\setminus W(x,2s)$, and so follows the claim.

It remains to control  the second half of the Hausdorff distance and the measure of the projection. For this purpose we furthermore assume that $x=f(t_0)$ is a   point of differentiability for $f$, which can be chosen $\Hh^1$-a.e. in $t\in [0,L]$ by Rademacher's theorem (\cite[Theorem 7.3.]{mat}), and therefore $\Hh^1$-a.e. in $x\in K$ (because the image of a $\Hh^1$-nullset by $f$ is still a $\Hh^1$-nullset in $K$).  It is clear that  $T_x=x+\mathrm{Span}(f'(t_0))$. Moreover there exists a $\xi(h)\to 0$ such that
\begin{eqnarray}
f(t_0+h)=f(t_0)+hf'(t_0)+\xi(h)h. \label{difff}
\end{eqnarray}
Let us identify $T_x\cap B(x,r)$ with the interval $[-r,r]$, and let $\delta>0$ be given. Then we chose $r_0$  small enough so that $|\xi(r)|\leq \delta/2$ for all $0<r\leq r_0$. It follows that for all $r\leq r_0$ and for every $t \in T_x\cap B(x,(1-\delta)r)$, thanks to \eqref{difff}, setting $y=f(t)$ we just have found a point $y\in K\cap B(x,r)$ such that $d(y,t)\leq \frac\delta 2r$. This implies that
$$\frac{1}{r}\sup \Big\{ d(t, K\cap B(x,r)) ; t \in T_x\cap B(x,r) \Big\}\leq \delta \quad \forall 0<r\leq r_0,$$
and since $\delta$ is arbitrary we have proved that
$$\lim_{r\to 0}\frac{1}{r}\sup \Big\{ d(t, K\cap B(x,r)) ; t \in T_x\cap B(x,r) \Big\}=0,$$
which together with \eqref{firsthalf} yields
$$\lim_{r\to 0}\frac{1}{r}d_H\Big(K\cap B(x,r),T_x\cap B(x,r)\big)=0.$$
We  finally prove \eqref{measproj}. Let $\varepsilon>0$ be fixed and let $r_0$ be small enough so that $|\xi(r)|\leq \varepsilon/10$ for all $r\leq r_0$. Let us set $\lambda_\varepsilon:=(1-\frac{11}{10}\varepsilon)\in (0,1)$. Due to  \eqref{difff}, we know that the curve $f([t_0,t_0+\lambda_\varepsilon r])$ stays inside the ball $B(x,r)$.  Since the image of a connected set by the continuous application $\pi$ is still connected, we deduce that $\pi(f([t_0,t_0+\lambda_\varepsilon r]))$ is an interval, that contains $x$ (identified with the origin on the line $T_x$) and contains $\pi(f(t_0+\lambda_\varepsilon r))$, which lies at distance less than $\varepsilon r/10$ from the point $(1-\frac{11}{10}\varepsilon r)$. We then conclude that
$$\pi(K\cap B(x,r))\supseteq \pi\big(f([t_0,t_0+\lambda_\varepsilon r])\big)\supseteq [f(t_0),f(\lambda_\varepsilon)r]\supseteq [0,(1-\varepsilon)r].$$
With a similar reasoning on the curve $f([t_0-\lambda_\varepsilon r,t_0])$ we get
$$\pi(K\cap B(x,r))\supseteq  [-(1-\varepsilon)r,(1-\varepsilon)r],$$
as desired.
\end{proof}
%%%%%%%%%%%%%%%%%%%%%

We will also need an easy covering Lemma, for which we give a detailed proof for the reader's convenience.

\begin{lem} \label{covering1} Let $A\subset \R^2$ be a compact set. Then for any $0\leq L<\Hh^1(A)$ there exists $r_0>0$ such that for every $r<r_0$ there exist a finite number of balls $B(x_i,r)$,  with $1\leq i \leq N$  satisfying the following properties:
\begin{eqnarray}
i) & & \text{the balls } B(x_i,r/2) \text{ are disjoint,} \notag \\
ii) & & x_i \in A  \text{ for all } i \text{ and } A \subset \bigcup_{i=1}^N B(x_i,r),  \notag \\
iii) & & L <  2Nr. \notag
\end{eqnarray}
\end{lem}
\begin{proof} Given $L< \Hh^1(A)$, 
%\begin{eqnarray}
%\varepsilon=\Hh^1(A)-L >0. \label{ineqqq1}
%\end{eqnarray}
define $\tau_0$ such that, for ever $\tau \in (0,\tau_0)$ we have
\begin{eqnarray}
L\leq \Hh_\tau^1(A)\leq \Hh^1(A). \label{ineqqq2}
\end{eqnarray}
Now by compactness of $A$, it is easy to find by induction  a finite number of points  $x_i \in A$, $1\leq i \leq N$, with the property that
$$A \subset \bigcup_{i=1}^N B(x_i,\tau /2),$$
and such that the $B(x_i,\tau /4)$ are disjoint. Since the family $\Big\{B(x_i,\tau/2)\Big\}_{i}$ is admissible in the  infimum  \eqref{infinmum}, we have that
\begin{eqnarray}
\Hh_\tau^1(A)\leq \sum_{i=1}^N\tau. \label{ineqqq3}
\end{eqnarray}
Gathering together \eqref{ineqqq2} and \eqref{ineqqq3} we obtain 
$$L<\sum_{i=1}^N\tau,$$
and setting $r_0=\tau_0/2$ achieves the proof.
%
%Now the case when $\Hh^1(A)=+\infty$ is treated the same way : let $L>0$ be arbitrary and let  $\tau>0$ be small enough so that
%$$\Hh_\tau^1(A)\geq 2L.$$
%If $x_i \in A$ is  a family of the same type as above with $1\leq i \leq N$ we have that
%$$\sum_{i=1}^N\tau \geq \Hh_\tau^1(A) \geq 2L > L,$$
%and the conclusion follows by setting $r=\tau/2$.
\end{proof}

%%%%%%%%%%%%%%%%%%%%%%%%%%%%%%%

\subsection{A topological lemma}

The following Lemma is quite obvious for arcwise connected sets, but in the sequel we will need to apply it in its full general version.

\begin{lem} \label{connexe} Let $A\subset \R^2$ be a compact and connected set,  $x\in A$ and $r \in (0,\frac 12\diam(A))$. Then the connected component of $\overline{B}(x,r)\cap A$ containing $x$, also contains a point of $A \cap \partial B(x,r)$.
\end{lem}
\begin{proof}   We shall use the following characterization of a connected set :
\begin{center}
\emph{(P) a compact metric space  $A$ is  connected if and only if it is well chained.}
\end{center}
 By ``well chained" we mean that for any two points $x,y \in A$ and for any $\varepsilon>0$, one can find an $\varepsilon$-chain of points $\{x_i\}_{i=1}^N\subset A$ such that $x_0=x$, $x_N=y$ and $d(x_{i+1},x_i)\leq \varepsilon$ for all $i$.
 A proof of the above claim can be found for instance in \cite[Proposition 19-3]{ch}.

 Now since $r \leq \diam(A)/2$, we can find a point $y \in A\setminus \overline{B(x,r)}$. Let   $\varepsilon>0$ be fixed. There exists an $\varepsilon$-chain $\{x_i\}$ that connects $x$ to $y$ in $A$. Let $x_\varepsilon$ be the last point of the chain inside $B(x,r)$ before the first exist, i.e. $x_\varepsilon := x_{i_0}$ where $i_0$ is defined as
 $$i_0:=\max \{i\;:\; x_j \in \overline{B}(x,r) \text{ for all } j\leq i \}.$$
Now we let $\varepsilon \to 0$. Precisely, for a subsequence $\varepsilon_n \to 0$ we can assume that $x_{\varepsilon_n}$ converges to some point $x_0$ inside the compact set $A\cap \overline{B}(x,r)$. It is easy to see that $x_0 \in \partial B(x,r)$.

 We claim moreover that $x_0 \in K$, where   $K\subset A\cap \overline{B}(x,r)$ is the connected component of $A\cap \overline{B}(x,r) $ that contains $x$. This will achieve the proof of the Lemma.

  To prove the claim, we introduce the notation $z \sim_A x $ to say that $z$ is well-chained with $x$ in $A$ (which means that for every $\varepsilon>0$ one can find an $\varepsilon$-chain of points  inside $A$ from $x$ to $z$) and we consider the set
 $$C := \{z  \in A \cap \overline{B}(x,r) \; ; \; z \sim_{A\cap \overline{B}(x,r)} x\}.$$
 Our goal is to prove that $C \subseteq K$. To see this we first notice that $C$ is closed because if $z_n$ is a sequence of points converging to some $z$, then for $n$ large enough (depending on $\varepsilon$) the $\varepsilon$-chain associated with $z_n$ will be admissible for $z$ too, with $z$ added at the end.   Thus $C$ is a compact set, and $C$ is naturally well-chained by its definition. Therefore $C$ is connected thanks to Property (P), and contains $x$. Therefore $C\subseteq K$.

To finish the proof it is enough to notice that $x_0$ naturally belongs to $C$ because it is a limit point of points $\varepsilon$-chained to $x$ in $A\cap \overline{B}(x,r)$, for arbitrary $\varepsilon >0$.
\end{proof}

%%%%%%%%%%%%%%%%%%%%%%%%%
\subsection{A lower bound on some geometric integral quantity}

For any closed set $A\subset \R^2$, $\lambda >0$ and direction $\nu \in \Ss$, we denote by $A_{\lambda,\nu}$ the $\lambda$-enlargment of $A$ in the direction $\nu$ defined by
$$A_{\lambda,\nu}:=\{ x+t \nu \; ;\; |t|\leq \lambda \text{ and } x \in A\}.$$
It is most obvious to check from the closeness of $A$ that $A_{\lambda,\nu}$ is also a closed set. %Indeed, assume by contradiction that $x \in A_{\lambda,\nu}^c$ and for all $\varepsilon>0$ we can find $x_\varepsilon \in A_{\lambda,\nu}\cap B(x,\varepsilon)$ such that $|x_\varepsilon-a_\varepsilon|< \lambda$ for some $a_\varepsilon \in A$. Then  after extracting a subsequence and letting $\varepsilon\to 0$ we would obtain that $x\in A_{\lambda,\nu}$, a contradiction.  
Next, we denote by $\mathscr{L}^2$   the two dimensional Lebesgue measure on $\R^2$ and consider for $\nu \in \Ss$ the function $\nu \mapsto \mathscr{L}^2(A_{\lambda,\nu})$. We claim that this function is $\Hh^1$-measurable. For this purpose let us prove that it is upper-semicontinuous. Indeed,  if $x\in \R^2$ is fixed and $\nu_{\varepsilon}\to \nu$ we can get from the closeness of $A_{\lambda,\nu}$ that
\begin{eqnarray}
\limsup_{\varepsilon} {\bf 1}_{A_{\lambda,\nu_\varepsilon}}(x)\leq {\bf 1}_{A_{\lambda,\nu}}(x). \label{uppsemicont}
\end{eqnarray}
Assume by contradiction that  $x \in \R^2 \setminus A_{\lambda,\nu}$, and that there exists a sequence $\nu_\varepsilon \to \nu$ for which $x\in A_{\lambda,\nu_\varepsilon}$ for all $\varepsilon$. Then there exists $a_\varepsilon \in A$ and $|t_\varepsilon|\leq \lambda$ such that $x=a_\varepsilon + t_\varepsilon \nu_\varepsilon $. Up to extract a subsequence, $t_\varepsilon$ converges to some $t$ satisfying $|t|\leq \lambda$ and $a_\varepsilon$ converges in $a$ to $A$, thus passing to the limit along this subsequence we obtain $x=a + t\nu \in A_{\lambda,\nu}$, a contradiction. Using \eqref{uppsemicont} and Fatou Lemma we deduce that $\nu \mapsto \mathscr{L}^2(A_{\lambda,\nu})$ is upper-semicontinuous, and therefore Borel-measurable. By consequence it is also $\Hh^1$-measurable and the following quantity is well defined.

\begin{eqnarray}
I_{\lambda}(A):=\frac{1}{2\lambda \pi}\int_{\Ss} \mathscr{L}^2(A_{\lambda,\nu}) d \Hh^{1}(\nu). \label{defIlambda}
\end{eqnarray}
We  summarize some elementary facts regarding $I_\lambda$ in the following Lemma.

\begin{lem} \label{Lemmaproj}For any direction $\nu \in \mathbb{S}^1$ we denote by $\pi_\nu$ the orthogonal projection onto the vectorial line directed by $\nu$. The following facts hold true.
\begin{enumerate}[i)]
\item For any closed set $A\subset \R^2$, $\nu\mapsto \mathcal{H}^1(\pi_\nu(E))$ is  $\mathcal{H}^1$-measurable on $\mathbb{S}^1$ and  we have
$$I_\lambda(A)\geq \frac{1}{\pi} \int_{\mathbb{S}^1}\mathcal{H}^1(\pi_\nu(E)) d\mathcal{H}^1(\nu).$$ 
%\item If $\lambda \leq \lambda'$ then {\color{red} do we use it ? and I think it's wrong} {\color{orange} tu as raison, il taut intervertir $\lambda$ et $\lambda'$ dans le quotient mais de toute façon on n'en a pas besoin}
%$$I_\lambda(A) \leq \frac{\lambda}{\lambda'}I_{\lambda'}(A).$$
\item If $\{E_k\}_{k \in \N}$ is a disjoint sequence of closed sets which satisfies
$$dist(E_k,E_{k'})\geq 2\lambda \quad \text{ for all } k\not = k'$$
then
$$I_{\lambda}(\bigcup_{k \in \N} E_k) = \sum_{k \in \N} I_{\lambda}(E_k).$$
\end{enumerate}
\end{lem}
\begin{proof} Assertions $ii)$ is quite clear directly from definitions, so only $i)$ needs a proof. We first remark that given a direction $\nu \in \mathbb{S}^1$ and $y\in \R$, then
$$\mathcal{H}^1\Big(A_{\lambda,\nu}\cap \pi_{\nu^\bot}^{-1}(y)\Big) \geq 
\left\{
\begin{array}{cl}
2\lambda  &\text{ if } A\cap \pi_{\nu^\bot}^{-1}(y) \not = \emptyset \\
0 & \text{ otherwise }
\end{array}
\right.
$$
Next, using Fubini's Theorem we can estimate
$$ \mathscr{L}^2(A_{\lambda,\nu}) = \int_{\R}\mathcal{H}^1\Big(A_{\lambda,\nu}\cap \pi_{\nu^\bot}^{-1}(y)\Big) dy\geq  2\lambda \int_{\R}\mathbf{1}_{ \pi_{\nu^\bot}(A)}(y)dy= 2\lambda \mathcal{H}^1(\pi_{\nu^{\bot}(A)}).$$
The measurability of $\nu \mapsto \mathcal{H}^1(\pi_{\nu^{\bot}(A)})$ is a delicate issue and follows from  \cite[2.10.5]{federer}. We finish the proof integrating over $\nu^\bot \in \mathbb{S}^1$.
\end{proof}

Our goal is now to prove the following.

\begin{lem} \label{Lemme2} Let $(A_n)_{n\in \N}$ be a sequence of  compact connected subsets of $\R^2$ converging for the Hausdorff distance to some compact and connected set $A$. Then for any $0\leq L<  \Hh^{1}(A) $  and for every $\lambda>0$ small enough (depending on $A$), one can find $n_0\in \N$ such that
$$I_\lambda(A_n)\geq C L, \quad \forall n \geq n_0$$
where $C>0$ is universal.
\end{lem}

\begin{proof} We can assume that $\diam(A)>0$ otherwise $L=0$ and there is nothing to prove. We start by applying Lemma \ref{covering1}, supposing $4\lambda<\min\{r_0,\diam(A)\}$. For $r=4\lambda$, we get the existence of some points $x_i \in A$ for $1\leq i\leq N$ such that the balls $B(x_i,r)$ satisfy the properties $i)$, $ii)$ and $iii)$ of  Lemma \ref{covering1}.

Find $n_0 \in \N$ large enough in such a way that
\begin{eqnarray}
d_H(A,A_n)\leq r.10^{-5} \quad \forall n\geq n_0 \label{Hausddist}
\end{eqnarray}
(here is where $n_0$ can depend on $\lambda$).
%Now we fix some $i$ between $1$ and $N$, and work in the ball  $B_i:=B(x_i,r/4)$. We claim that with our choice of $\lambda$ and $n_0$ it holds
%\begin{eqnarray}
%\frac{1}{2\lambda\pi}\int_{\Ss} \mathscr{L}^2((A_n)\cap B_i)_{\lambda,\nu} d \Hh^{1}(\nu) \geq C r, \quad \forall n\geq n_0, \label{BIGCLAIM}
%\end{eqnarray}
%where $C>0$ is universal.
%
Given $n\geq n_0$, let $z_{i,n} \in A_n$ be a point such that $d(z_{i,n},x_i)\leq r.10^{-5}$ (this point exists due to \eqref{Hausddist}). Since $r<\diam(A)$ there exists a point in $A\setminus \overline{B}(x_i,r)$, and applying \eqref{Hausddist} again there exists a point of $A_n$ close to this one so that $A_n\setminus \overline{B}(x_i, r(1-10^{-5}))$ is not empty, containing at least a point $z'_{i,n}$. This proves that $\diam(A_n)\geq d(z_{i,n},z'_{i,n})> \frac{1}{2}r$. But then we can apply Lemma \ref{connexe} to the connected set $A_n$ in $B(z_{i,n},\frac{1}{4}r)$, and we get
\begin{eqnarray}
K_{i,n}\cap \partial B\left(z_{i,n},\frac{r}{4}\right) \not = \emptyset, \notag
\end{eqnarray}
where $K_{i,n}\subset A_n$ is the connected component of $A_n\cap \overline{B(z_{i,n},r/4)}$ that contains $z_{i,n}$. Let $y_{i,n} \in K_{i,n}\cap \partial B\left(z_{i,n},\frac{r}{4}\right) $.

We now claim that with our choice of $\lambda$ and $n_0$ it holds
\begin{eqnarray}
I_\lambda(A_n) \geq
\sum_{i=1}^N I_\lambda(K_{i,n})\geq C L, \quad \forall n\geq n_0, \label{BIGCLAIM}
\end{eqnarray}
%\begin{eqnarray}
%\frac{1}{2\lambda\pi}\int_{\Ss} \mathscr{L}^2((A_n)_{\lambda,\nu}) d \Hh^{1}(\nu) \geq
%\sum_{i=1}^N \frac{1}{2\lambda\pi}\int_{\Ss} \mathscr{L}^2((K_{i,n})_{\lambda,\nu}) d \Hh^{1}(\nu) \geq C L, \quad \forall n\geq n_0, \label{BIGCLAIM}
%\end{eqnarray}
where $C>0$ is universal.

To justify the first inequality, notice that the sets $(K_{i,n})_{\lambda,\nu}$ are all disjoint since $\lambda=r/4$ and $K_{i,n}\subset B_i=B(x_i,r/4)$, which provides $((K_{i,n})\cap B_i)_{\lambda,\nu}\subset B(x_i,r/2)$, and these balls are disjoint. 

To estimate the integral with $K_{i,n}$, we use the first assertion in Lemma \ref{Lemmaproj}, after noticing that, for each projection $\pi_\nu$, if $E$ is connected and $\{a,b\}\subset E$, then $\pi_\nu(E)\supset\pi_\nu([a,b])$. This means that we have 
$$I_\lambda(E)\geq I_\lambda([a,b])=\frac{2}{\pi}|a-b|.$$
In our particular case we can take $E=K_{i,n}$, $a=z_{i,n}$ and $b=y_{i,n}$ and get $I_\lambda(K_{i,n})\geq I_\lambda([z_{i,n}-y_{i,n}])\geq Cr$.

Hence, summing up and using  Lemma \ref{Lemmaproj} we get
$$I_\lambda(A_n)\geq \sum_{i=1}^N I_\lambda((K_{i,n})\cap B_i)\geq C  \sum_{i=1}^N r,\geq C L,$$
as desired. This finishes the proof of the Lemma.
\end{proof}

\subsection{An elementary inequality on the  total variation}
%%%%%%%%%%%%%%%%%%%%%
When $I=(a,b)\subset \R$ is an interval, we denote by $Var(f,I)$ the total variation of the one-variable function $f:I\to \R$ defined by
$$Var(f,I):=\sup \{\sum_i |f(t_i)-f(t_{i+1})| ; \quad a=t_0<t_1<\dots <t_i<\dots <t_N=b \}.$$
If $J \subset \R$ is open, we define $Var(f,J)=\sum_I Var(f,I)$, where the sum is taken over all the connected components of $J$.
\begin{lem}\label{variation} Let $J\subset \R$ be an open set and $f: J \to \R$ a Borel measurable function. Take a finite number of intervals $I^\pm_i$,  $1\leq i\leq N$, satisfying the following properties:
\begin{enumerate}[(i)]
\item $\bar I^-_i=[a_i^-,c_i]$ and $\bar I^+_i=[c_i,a_i^+]$ with $a_i^-< c_i< a_i^+$
\item $c_i < c_{i+1}$ for all $1\leq i \leq N$
\item $J=\bigcup_{1\leq i\leq N} I^-_i \cup I^+_i.$
\end{enumerate}
Then denoting by $m_i^\pm$ the average of $f$ on $I_i^\pm$ we have that
$$Var(f,J)\geq \sum_{1\leq i\leq N} m_i^+ + m_i^- -2 f(c_i).$$
\end{lem}

\begin{proof} This Lemma is very elementary but we provide a detailed proof for the convenience of the reader. Observe first that   for all $1\leq i \leq N$,  $I_i^-\cup I_i^+ $ is an interval and therefore is contained in only one connected component of $J$.  Therefore, arguing  component by component we can assume without loss of generality that  $J$ is connected.

We start by writing
\begin{eqnarray}
\sum_{1\leq i\leq N} m_i^+ + m_i^- -2 f(c_i)\leq \sum_{1\leq i\leq N} |m_i^+ - f(c_i)|+|m_i^--f(c_i)|. \label{inequalityR}
\end{eqnarray}
Next, for each $1\leq i \leq N$ we define $t_i^+ \in I^+_i$ and $t_i^- \in I^-_i$ as follows :
\begin{itemize}
\item if $m_i^+ -f(c_i)\geq 0$ then we select some $t_i^+ \in Int ( I^+_i)$ such that $f(t_i^+)\geq m_i^+$. In this case we have that
$$|m_i^+ -f(c_i)|=m_i^+ -f(c_i)\leq f(t_i^+)-f(c_i)=|f(t_i^+)-f(c_i)|.$$
\item if, at the contrary, $m_i^+ -f(c_i)\leq 0$ then we select some $t_i^+ \in Int ( I^+_i)$ such that $f(t_i^+)\leq m_i^+$. It follows in this case that
$$|m_i^+ -f(c_i)|=f(c_i)-m_i^+ \leq f(c_i)-f(t_i^+)=|f(t_i^+)-f(c_i)|.$$
We proceed the same way for $m_i^-$, defining for all $1\leq i \leq N$ a point  $t_i^- \in Int ( I^-_i)$ satisfying
$$|m_i^- -f(c_i)|\leq |f(t_i^-)-f(c_i)|.$$
By this way we have obtained a subdivision of $J$ of the form
$$a_1^- < t_1^- < c_1 < t_1^+ < \dots < t_i^- < c_i < t_i^+ < \dots a_N^+.$$
Consequently
$$\sum_{1 \leq i\leq N}|f(c_i)-f(t_i^-)|+|f(t_i^+)-f(c_i)|+|f(t_{i+1}^-)-f(t_i^+)|\leq Var(f,J)$$
And by construction of $t_i^\pm$ it follows that
\begin{eqnarray}
 \sum_{1\leq i\leq N} |m_i^+ - f(c_i)|+|m_i^--f(c_i)| &\leq& \sum_{1 \leq i\leq N}|f(c_i)-f(t_i^-)|+|f(t_i^+)-f(c_i)| \notag \\
 &\leq & Var(f,J), \notag
 \end{eqnarray}
and the proof is concluded by recalling \eqref{inequalityR}.\qedhere
\end{itemize}
\end{proof}
%{\color{red} on est s\^urs qu'il faut d\'etailler \`a ce point l\`a ?}
%%%%%%%%%%%%%%%%%%%%%%%%%%%%%%%%%%%%%%%%
\subsection{A useful recovery sequence}

Here we recall a standard construction for the $\Gamma$-limsup inequality, that is obtained by studying the optimal profile in one dimension for minimizers of an energy of Modica-Mortola type.

\begin{lem}\label{limsup} Let $\Omega \subset \R^2$ be open and $K\subset \Omega$ be a compact and connected set. For any $r>0$ we will set
$K_r:=\{x\in\R^2 ; dist(x,K)\leq r\}.$
 Let $k_\varepsilon\to 0$ be given. Then for all $\varepsilon>0$ there exists a function $\varphi_\varepsilon \in W^{1,2}_{loc}(\R^2)\cap C^0(\R^2)$ equal to $k_\varepsilon$ on $K_{\varepsilon^2}$, equal to $1$ on $\R^2 \setminus K_{\varepsilon^2+2\varepsilon|\ln(\varepsilon)|}$, and such that 
$$\limsup_{\varepsilon \to 0} \left( \int_{\R^2}\varepsilon |\nabla \varphi_\varepsilon|^2 + \frac{(1-\varphi_\varepsilon)^2}{4\varepsilon} \right) \leq \Hh^{1}(K)$$
(the integral being indeed performed on a neighborhood of $\Omega$, since $\varphi_\ve$ is the constant $1$ outside).
\end{lem}

\begin{proof} Due to the fact that $K$ is connected, it automatically satisfies the lower Ahlfors-regularity inequality
\begin{eqnarray}
\inf_{x\in K, r<r_0} \frac{\Hh^1(K\cap B(x,r))}{2r} > 0.
\end{eqnarray}
But this is enough to guarantee that the Minkowsky content and the Hausdorff measure coincide (see for e.g. \cite[Theorem 2.104, page 110]{afp}), namely,

\begin{eqnarray}\label{assumpt0}
\lim_{r \to 0} \frac{\mathscr{L}^{2}(
K_r)}{2r}=\Hh^{1}(K).
\end{eqnarray}
Then, for some suitable infinitesimals $a_\varepsilon$,
$b_\varepsilon$  that will be chosen just after, we take a function   $\varphi_\varepsilon$ similar to the one considered  in the proof of Theorem 3.1. in \cite{at}, namely
$$
\varphi_\varepsilon =
\left\{
\begin{array}{cl}
k_\varepsilon & \text{ on } K_{b_\varepsilon}\\
1 &\text{ on } \Omega \setminus K_{b_\varepsilon+a_\varepsilon} \\
k_\varepsilon+\lambda_\ve\left[1-\exp\left(\frac{b_\varepsilon-dist(x,K)}{2\varepsilon}\right) \right]&
\text{ on } K_{b_\varepsilon+a_\varepsilon}\setminus
K_{b_\varepsilon}
\end{array}
\right.
$$
where 
$$\lambda_\varepsilon:=\frac{1-k_\varepsilon}{1-\exp(-a_\ve/(2\ve))},$$
which insures that $\varphi_\varepsilon$ is a continuous function.  We will impose $a_\ve/\ve\to+\infty$, so that $\lambda_\ve\to 1$. Precisely, similarly to what was performed in \cite{at} page 116-117, we can show that,  by
choosing $b_\varepsilon = \varepsilon^2$ and $a_\varepsilon=-2\varepsilon \ln(\varepsilon)$,  it holds 
$$\limsup_{\varepsilon \to 0} \left( \int_{U}\varepsilon |\nabla \varphi_\varepsilon|^2 + \frac{(1-\varphi_\varepsilon)^2}{4\varepsilon} \right) \leq \Hh^{1}(K).$$
For the convenience of the reader, let us write here the proof.

Notice that the choice of $a_\ve$ implies $\lambda_\ve=(1-k_\ve)/(1-\ve)$.

The main contribution in the limsup is attained in the region $K_{a_\varepsilon+b_\varepsilon}\setminus K_{b_\varepsilon}$, since it is easily seen that all the rest goes to zero. Denoting  $\tau(x):=dist(x,K)$, we observe that in this region  the function $\varphi_\varepsilon$ is of the form $\lambda_\ve f(\tau(x))+k_\varepsilon$,  where $f(t)=[1-\exp\left(\frac{b_\varepsilon-t}{2\varepsilon}\right)]$ is solving  the equation
$$f'=\frac{1-f}{2\varepsilon}\quad, \quad f(b_\varepsilon)=0.$$
Also notice that for $t\in[b_\ve,b_\ve+a_\ve]$ we have $f(t)\leq 1-\ve$. We also have (we will use this computation in the next inequalities)
$$1-k_\ve-\lambda_\ve f(t)=(1-k_\ve)\left(1-\frac{f(t)}{1-\ve}\right)=\frac{1-k_\ve}{1-\ve}(1-\ve-f(t)).$$ 
Also, from $f(t)\leq 1-\ve$, we infer $(1-k_\ve-\lambda_\ve f(t))^2\leq \frac{(1-k_\ve)^2}{(1-\ve)^2}(1-\ve-f(t))^2$.

 The coarea formula yields 
\begin{eqnarray}
A_\varepsilon&:=&\int_{K_{a_\varepsilon+b_\varepsilon}\setminus K_{b_\varepsilon}}\varepsilon |\nabla \varphi_\varepsilon|^2 + \frac{(1-\varphi_\varepsilon)^2}{4\varepsilon} dx\notag \\
&=&\int_{b_\varepsilon}^{b_\varepsilon+a_\varepsilon} \left( \varepsilon \lambda_\ve^2 |f'(t)|^2 + \frac{(1-k_\varepsilon-\lambda_\ve f(t))^2}{4\varepsilon} \right) \mathcal{H}^1\Big(\{ x ; dist(x,K)=t\Big) dt \notag \\ 
 \notag 
 &\leq&\int_{b_\varepsilon}^{b_\varepsilon+a_\varepsilon} \left(\lambda_\ve^2 \frac{(1-f(t))^2}{4\varepsilon} +  \frac{(1-k_\ve)^2}{(1-\ve)^2}\frac{(1-f(t))^2}{4\varepsilon} \right) \mathcal{H}^1\Big(\{ x ; dist(x,K)=t\Big) dt \notag \\ 
 \notag 
 &\leq&\max\left\{\lambda_\ve^2,  \frac{(1-k_\ve)^2}{(1-\ve)^2}\right\}\frac{1}{2\varepsilon}\int_{b_\varepsilon}^{b_\varepsilon+a_\varepsilon}\exp\left(\frac{b_\varepsilon-t}{\varepsilon}\right)\mathcal{H}^1\Big(\{ x ; dist(x,K)=t\Big) dt .
\end{eqnarray}
setting
$$\sigma_\ve:=\max\left\{\lambda_\ve^2,  \frac{(1-k_\ve)^2}{(1-\ve)^2}\right\}\to 1$$
we only need to estimate the integral in the last expression.

Denoting $g(t)=\mathscr{L}^{2}(
K_t)$, we have $g'(t)=\mathcal{H}^1\Big(\{ x ; dist(x,K)=t\Big)$ a.e. thus after integrating by parts we arrive to
\begin{eqnarray}
A_\varepsilon &\leq&\frac{\sigma_\varepsilon\varepsilon}{2}g(a_\varepsilon+b_\varepsilon)-\frac{\sigma_\varepsilon}{2\varepsilon}g(b_\varepsilon) + \frac{\sigma_\varepsilon}{2\varepsilon^2}\int_{b_\varepsilon}^{b_\varepsilon+a_\varepsilon}\exp\left(\frac{b_\varepsilon-t}{\varepsilon}\right)g(t) dt . \notag 
\end{eqnarray}
Now since the first two terms are going to zero, and since $\sigma_\varepsilon\to 1$, we get
\begin{eqnarray}\label{above vareps}
\limsup_{\varepsilon \to 0} A_\varepsilon \leq \limsup_{\varepsilon \to 0} \frac{1}{2\varepsilon^2}\int_{b_\varepsilon}^{b_\varepsilon+a_\varepsilon}\exp\left(\frac{b_\varepsilon-t}{\varepsilon}\right)g(t) dt .
\end{eqnarray}
If we compute, with the change of variable $t=b+\ve s$
$$\frac{1}{2\ve^2}\;\int_b^{a+b}\exp\left(\frac{b-t}{\ve}\right)2t dt =\int_0^{a/\ve}\exp(-s)\left(\frac b \ve +s\right)ds$$
and we use $a=a_\ve=-2\ve\ln(\ve)$ and $b=b_\ve=\ve^2$ we get $a_\ve/\ve\to +\infty$ and $b_\ve/\ve\to 0$, which gives
\begin{equation}\label{chg var}
\lim_{\ve\to 0}\;\frac{1}{2\ve^2}\int_b^{a_\ve+b_\ve}\exp\left(\frac{b_\ve-t}{\ve}\right)2t dt=\int_0^{\infty}\exp(-s)sds=1.
\end{equation}

Then, we can deduce from the computation above and from \eqref{assumpt0}  that the right-hand side of \eqref{above vareps} does not exceed $\mathcal{H}^1(K)$.  Indeed, for any given $\tau>0$ we know from \eqref{assumpt0}  that $g(t)/2t$ is less than  $\mathcal{H}^1(K)+\tau$, provided that $\varepsilon$ is small enough. This allows us to write, for small $\varepsilon$,
$$
 \frac{1}{2\varepsilon^2}\int_{b_\varepsilon}^{b_\varepsilon+a_\varepsilon}\exp\left(\frac{b_\varepsilon-t}{\varepsilon}\right)g(t) dt \leq  \frac{(\mathcal{H}^1(K)+\tau)}{2\varepsilon^2}\int_{b_\varepsilon}^{b_\varepsilon+a_\varepsilon}\exp\left(\frac{b_\varepsilon-t}{\varepsilon}\right) \,2t\, dt\notag
% &\leq & \frac{(\mathcal{H}^1(K)+\tau)(b_\varepsilon+a_\varepsilon)}{\varepsilon^2}\int_{b_\varepsilon}^{b_\varepsilon+a_\varepsilon}\exp\left(\frac{b_\varepsilon-t}{\varepsilon}\right) \, dt \notag \\
% &=&  \frac{(\mathcal{H}^1(K)+\tau)(b_\varepsilon+a_\varepsilon)(1-\varepsilon)}{\varepsilon} \notag 
 %\end{eqnarray}
$$
and we conclude by taking the limsup in $\varepsilon$, using \eqref{chg var}, and then letting $\tau \to 0$.

\end{proof}

\begin{remark}\label{estimate ve-p}
Notice that from the explicit formulas for $\varphi_\ve$, one can also estimate higher summability for $|\nabla\varphi_\ve|$, with bounds depending on $\ve$. Without any intention to be sharp, it is easy to check $ |\nabla\varphi_\ve|\leq C/\ve$, and hence $\int |\nabla\varphi_\ve|^p\leq C\ve^{-p}$.
\end{remark}

%{\color{red} ce que je mets ci-dessous sert \`a r\'egler le probl\`eme avec le bord. cependant, je sugg\`ere de le r\'egler une fois pour toutes en imposant $\phi=1$ sur le bord. dans ce cas l\`a les consid\'erations ci-dessous peuvent \^etre adapt\'ees pour satisfaire cette contrainte.}

 The functions $\varphi_\ve$ that we have just built have the property that they are equal to $\varepsilon$ on  $K$ and are equal to $1$ outside a neighborhood of $K$. In  the sequel we will need to ``move'' a little bit  the set $K$ inside $\Omega$ to take care of boundary effects.

In particular we will consider the following situation.

\begin{lem}\label{limsupmodifie}Let  $\Omega$ be an open  and bounded subset of $\R^2$, star-shaped with respect to $0\in \Omega$ and with Lipschitz boundary.   Then there exists some constants $c(\Omega)$ and $C(\Omega)$  depending only on $\Omega$ and expliciten later, such that the following holds. Take $a_\varepsilon$, $b_\varepsilon$ and $\varphi_\varepsilon$ just as before for some closed connected set $K \subset \overline\Omega$,   and with $\varepsilon$ satisfying $a_\varepsilon+b_\varepsilon\leq c(\Omega)$. Let $\delta\in]0,1[$ be such that
\begin{eqnarray}
\delta \geq  C(\Omega)(a_\varepsilon+b_\varepsilon).\label{choiceofdelta}
\end{eqnarray}
Then 
\begin{eqnarray}
\frac{1}{1+\delta}K_{a_\varepsilon+b_\varepsilon}\subset \Omega \label{cool0}
\end{eqnarray}
and the function $\varphi_{\ve,\delta}$ defined by
$$\varphi_{\ve,\delta}(x)=\begin{cases}\varphi_\ve((1+\delta)x)&\mbox{if }x\in(1+\delta)^{-1}\Omega,\\
								1&\mbox{if }x\in \Omega\setminus (1+\delta)^{-1}\Omega,\end{cases}$$
satisfy
$$\left( \int_{\Omega}\varepsilon |\nabla \varphi_{\varepsilon,\delta}|^2 + \frac{(1-\varphi_{\varepsilon,\delta})^2}{4\varepsilon} \right) \leq \left( \int_{\Omega}\varepsilon |\nabla \varphi_\varepsilon|^2 + \frac{(1-\varphi_\varepsilon)^2}{4\varepsilon} \right).$$
 The constants $c(\Omega)$ and $C(\Omega)$ can be chosen as follows
$$c(\Omega)=\frac{1}{4}\min\{|x| ;\; x\in \partial \Omega\}, \quad \quad C(\Omega)=\frac{1}{4c(\Omega)}\left(1+\frac{8{\rm Lip}(g_\Omega)}{c(\Omega)}\right),$$
where $g_\Omega$ is the Lipschitz parametrisation of the boundary defined through the identity
 $$\Omega=\{x ; |x|\leq g_\Omega(x/|x|)\}.$$

\end{lem}
\begin{proof}   Let $g=g_\Omega$ be the gauge of $\Omega$ and $c(\Omega)$, $C(\Omega)$ the constants given in the statement of the Lemma.  We want to find a condition on $\delta$ such that, for every point $x\in\Omega$ and every vector $w$ with $|w|\leq a_\ve+b_\ve \leq c(\Omega)$, we have $(1+\delta)^{-1}(x+w)\in \Omega$. 
  Notice that $B(0,R)\subset \Omega$ where $R:=\min\{|x| ;\; x\in \partial \Omega\}$, and since  $a_\varepsilon+b_\varepsilon\leq c(\Omega)=R/4$, it follows  that $x+w \in B(0,R)\subset \Omega$ for all $x \in B(0,R/2)$, a fortiori  $(1+\delta)^{-1}(x+w)\in \Omega$  too. Therefore it is enough to check the property for $x \in \Omega \setminus B(0,R/2)$.

From the inequality 
$$\left|\frac{x}{|x|}-\frac{y}{|y|}\right|=\left|\frac{|y|(x-y)+y(|y|-|x|)}{|x||y|}\right|\leq \frac{|x-y|+\big||x|-|y|\big|}{|x|}\leq \frac{2|x-y|}{|x|}$$
 we deduce that 
\begin{eqnarray}
\left|\frac{x}{|x|}-\frac{x+w}{|x+w|}\right| \leq \frac{2|w|}{|x|}.
\end{eqnarray}
It follows that 
\begin{eqnarray}
g\left(\frac{x}{|x|}\right)\leq g\left(\frac{x+w}{|x+w|}\right)+{\rm Lip}(g)\frac{2w}{|x|} .\label{liplip}
\end{eqnarray}

Recall that we want to find a condition on $\delta$ which garantees 
$$|x+w|\leq (1+\delta)g\left(\frac{x+w}{|x+w|}\right).$$
 From $|x+w|\leq |x|+|w|$ and \eqref{liplip}  it is enough to have 
\begin{equation}\label{lipg}
|x|+|w|\leq g\left(\frac{x}{|x|}\right)+\delta g\left(\frac{x}{|x|}\right)-(1+\delta)\lip(g)\frac{2|w|}{|x|}.\end{equation}
Using $|x| \leq g(x/|x|)$ and $\delta\leq 1$, \eqref{lipg} is guaranteed if 

$$|w|\leq \delta g\left(\frac{x}{|x|}\right)-4\lip(g)\frac{|w|}{|x|},$$
in other words when
$$\delta\geq \frac{|w|}{g(x/|x|)}\left(1+\frac{4\lip(g)}{|x|}\right),$$
and since $|w|\leq a_\varepsilon+b_\varepsilon$, $|x|\geq R/2=2c(\Omega)$ and $g(x/|x|)\geq 4c(\Omega)$ this holds true when
$$\delta\geq (a_\varepsilon+b_\varepsilon)\frac{1}{4c(\Omega)}\left(1+\frac{8\lip(g)}{c(\Omega)}\right),$$

which gives a choice for $C(\Omega)$.

We shall prove the last conclusion.  The change-of variable $y=(1+\delta)x$ shows
$$\int_U |\nabla\varphi_{\ve,\delta}|^2=\!\int_{(K_{a_\varepsilon+b_\varepsilon})/(1+\delta)}\!\!\!|(1+\delta)\nabla\varphi_\ve((1+\delta)x)|^2=\!\int_{(K_{a_\varepsilon+b_\varepsilon})}\!\!\!|\nabla\varphi_\ve(x)|^2=\!\int_\Omega |\nabla\varphi_{\ve}|^2.$$
of course this is not surprising because our change of variable is conformal. Moreover, by use of the same change of variable and using the fact that $(1-\varphi_{\varepsilon,\delta})^2 =0$ on  $ \Omega\setminus (1+\delta)^{-1}\Omega$ we also have that
$$\int_{\Omega}(1-\varphi_{\ve,\delta})^2dx=\int_{(1+\delta)^{-1}\Omega}(1-\varphi_{\ve,\delta})^2dx= \frac{1}{(1+\delta)^2} \int_{\Omega} (1-\varphi_\ve)^2dx \leq   \int_{\Omega} (1-\varphi_\ve)^2dx.$$
%{\color{orange}It remains to prove \eqref{zerozero}. First observe that by definition of  $\varphi_{\varepsilon,\delta}$ we know that it vanishes on the set $(1+\delta)^{-1}K_{b_\varepsilon}$. On the other hand  the following formula holds true
%$$(1+\delta)^{-1}K_{b_\varepsilon}=[(1+\delta)^{-1}K]_{\frac{b_\varepsilon}{1+\delta}}.$$
%Now we claim that under the condition $\delta\leq \min(1,\frac{b_\varepsilon}{4\diam(\Omega)})$, then
%$$[(1+\delta)^{-1}K]_{\frac{b_\varepsilon}{1+\delta}}\supseteq K_{\frac{b_\varepsilon}{4}},$$
%which is enough to conclude. To prove the claim, consider any point $y \in K_{\frac{b_\varepsilon}{4}}$, and let $\bar y \in K$ be such that $\modvec y-\bar y\modvec \leq b_{\varepsilon}/4$.  Then $\bar y' := (1+\delta)^{-1}\bar y \in (1+\delta)^{-1} K$ and 
%$$\modvec y-\bar y'\modvec \leq \modvec y-\bar y\modvec  + \modvec \bar y - \bar y'\modvec \leq \frac{b_\varepsilon}{4}+ \modvec \bar y\modvec \frac{\delta}{1+\delta}\leq  \frac{b_\varepsilon}{4}+\diam(\Omega)\delta \frac{1}{2}\leq  \frac{3b_\varepsilon}{8}\leq \frac{b_\varepsilon}{(1+\delta)}$$
%which means that $y \in [(1+\delta)^{-1}K]_{\frac{b_\varepsilon}{1+\delta}}$ and finishes the proof of the lemma.
\end{proof}

\begin{remark} The same argument could be adapted in more general domains than star-shaped by use of  other transformations $T_\delta$ instead of $y=(1+\delta)x$, involving coefficients depending on $\det DT_\delta$ and $||DT_\delta||$, quantities which are close to $1$ if $DT_\delta$ is close to the identity matrix. 
\end{remark}

%{\color{red} Il faudra r\'e\'ecrire mieux tout cela : veut-on se limiter au cas \'etoil\'e ? en fait, cela demande de la r\'egularit\'e en plus : on a besoin que $\Omega/(1+\delta)$ soit \`a distance au moins $c\delta$ de $\partial\Omega$ (donc, domaine Lipschitz)}

%%%%%%%%%%%%%%%%%%%%
\section{The main  liminf inequality}

We are now ready to establish the main inequality that will lead to our $\Gamma$-convergence result. %We recall that if $\Omega$ is a domain with Lipschitz boundary, then it is an extension domain for the Sobolev space $W^{1,1}$. This means that there exists a bounded extension operator $E : W^{1,1}(\Omega) \to W^{1,1}(\R^2)$ such that $E(u)=u$ in $\Omega$. When $\Omega$ is bounded we can furthermore assume that $E(u)$ has a compact support in $\R^2$ and that it has the following properties :
%\begin{eqnarray}
%& & \text{ If } u  \in [0,1]  \text{ in } \Omega \text{ then } E(u) \in [0,1] \text{ in } \R^2.  \label{proprieteL2}  \\
%& &  \text{ If } u_\varepsilon \to 1 \text{ in } L^1(\Omega) \text{ then } E(u_\varepsilon)\to 1 \text{ in } L^1(\Omega'),  \label{proprieteL1}
%\end{eqnarray}
%where $\Omega':=\{x ; d(x,\Omega)\leq 1\}$. We refer to \cite{Jones} for more information on extension domains.

\begin{lem} \label{mainlemma} Let $\Omega\subset \R^2$ be a  bounded open set, $\mu_\varepsilon$ be a family of measures on $\Omega$ and  $\varphi_\varepsilon \in H^1(\Omega)\cap C(\overline\Omega)$ be a family of functions satisfying $\varphi_\varepsilon=1$ on $\partial \Omega$ and $0\le \varphi_\ve\le 1$ in $\Omega$, $c_\varepsilon\to 0$ and  $x_\varepsilon$ a sequence of points of $\Omega$ such that
\begin{enumerate}[{\rm ($i$)}]
\item $x_\varepsilon \to x_0$ for some $x_0 \in \overline{\Omega}$
\item $\mu_\varepsilon \destar \mu$ for some measure $\mu$ on $\Omega$
\item $d_{\varphi_\varepsilon}(x,x_\varepsilon)$ converges uniformly to some 1-Lipschitz function $d(x)$ on $\Omega$.
\item $$\sup_{\varepsilon>0} \left( \frac{1}{4\varepsilon}\int_{\Omega}(1-\varphi_\varepsilon)^2 dx + \varepsilon\int_{\Omega}\modvec \nabla \varphi_\varepsilon \modvec ^2 dx +\frac{1}{c_\varepsilon}\int_{\Omega}d_{\varphi_\varepsilon}(x,x_\varepsilon) d\mu_\varepsilon \right)\leq C<+\infty.$$
\end{enumerate}
Then
\begin{enumerate}[{\rm ($a$)}]
\item  the compact set $K:=\{x \in \overline{\Omega} \; ;\; d(x)=0\}$ is  connected,
\item $x_0\in K$,
\item $\spt(\mu)\subset K$,
\item we have
\begin{eqnarray}
\Hh^1(K) \leq \liminf_{\varepsilon\to 0} \left( \frac{1}{4\varepsilon}\int_{\Omega}(1-\varphi_\varepsilon)^2 dx + \varepsilon\int_{\Omega}\modvec \nabla \varphi_\varepsilon \modvec ^2 dx \right).\notag
\end{eqnarray}
\end{enumerate}
\end{lem}

%{\color{red}
%{\bf je me demande s'il ne faudrait pas d\'efinir nos pbs de minimisation avec $\varphi=1$ sur le bord et d\'emontrer qu'\`a la limite on a toute la mesure $\Hh^1$ (y compris le bord)}}

\begin{proof}  Up to extracting a subsequence in $\varepsilon$ we may assume that the liminf in $(d)$ is a true limit. For the sake of simplicity, in the sequel we will continue to denote by $\varepsilon$ the index of this subsequence, and it will still be the same for any further subsequences in the future.

For every $\varepsilon$ and $\delta>0$ we define
$$K_{\varepsilon,\delta}:=\{x\in \overline{\Omega} \;;\;d_{\varphi_{\varepsilon}}(x,x_\varepsilon)\leq \delta\}.$$ 
Since these sets are all compact sets contained in $\overline{\Omega}$, up to a subsequence we can assume that
$K_{\varepsilon,\delta}$ converges to some $K_\delta$ for the Hausdorff distance when $\varepsilon \to 0$. Next we define
$$K':=\bigcap_{\delta>0} K_\delta.$$
Notice that for  any pair of points $x,y \in K_{\varepsilon,\delta}$, the geodesic curve that realized the  distance $d_{\varphi_{\varepsilon}}(x,y)$  connecting $x$ to $y$ is totally contained in $K_{\varepsilon,\delta}$, thus $K_{\varepsilon, \delta}$ is  path connected. Consequently, $K_\delta$ is connected as Hausdorff limit of connected sets, and therefore $K'$ is also connected as  a decreasing intersection of connected sets (which is also a Hausdorff limit). But since $d_{\varphi_\varepsilon}(x,x_\varepsilon)$ converges uniformly to some 1-Lipschitz function $d(x)$, it is easy to see that $K'=K$. Indeed, for each $\delta>0$ we have $\{d<\delta\}\subset K_\delta\subset\{d\leq \delta\}$, which implies $K'=\{d=0\}$. In particular, since $d_{\varphi_\varepsilon}(x_0,x_\varepsilon)\leq |x_\ve-x_0| \to 0$, we obtain  that $d(x_0)=0$, thus $x_0 \in K$.

The next part of the claim, i.e. $\spt(\mu) \subseteq K$
%\begin{equation}
%\spt(\mu) \subseteq K .\label{claim10}
%\end{equation}
is an easy consequence of the fact that 
$$\int_{\Omega} d_{\varphi_\varepsilon}(x,x_\varepsilon)d\mu_\ve\to 0,$$
which gives, thanks to the uniform convergence $d_{\varphi_\varepsilon}(x,x_\varepsilon)\to d(x)$ and the weak convergence of the measures, $\int_{\Omega}d(x)\,d\mu(x)=0$. Since $d$ is a continuous function, we get $d=0$ on $\spt(\mu)$ and hence $\spt(\mu)\subset K$.

At this stage $(a), (b), (c)$ are proved and it remains to prove $(d)$. This will be achieved in two main steps.

\vspace{0.5cm}
\noindent \emph{Step 1. Rectifiability of $K$.} We first prove that
\begin{eqnarray}
 \mathcal{H}^1(K) < +\infty .\label{limiinf0}
\end{eqnarray}

Fix $\delta_0,\tau_0>0$, and let $\{z_1, z_2, \dots, z_N\}\subseteq K$ be a $\tau_0$-network in $K$, which means
 $$K\subseteq  \bigcup_{1\leq i\leq N} B(z_i,\tau_0).$$
Due to the uniform convergence $d_{\varphi_\varepsilon}(\cdot,x_\varepsilon)\to d$ and the fact that $d(z_i)=0$ for all $1\leq i\leq N$, there exists an $\varepsilon_1>0$, depending on $\delta_0$ and $\tau_0$, such that the following holds : for any $\varepsilon<\varepsilon_1$ there exists some $C^1$ regular curves $\Gamma_i^\varepsilon$ (of finite length) connecting $z_i$ to  $x_\varepsilon$ and satisfying
\begin{eqnarray}
\int_{\Gamma_i^\varepsilon} \varphi_\varepsilon(s)d\Hh^1(s) < \delta_0,\quad \forall 1\leq i \leq N. \label{courbespetites}
\end{eqnarray}
Now we consider
$$\Gamma_\varepsilon := \bigcup_{1\leq i \leq N} \Gamma_i^\varepsilon,$$
(which also depends on $\delta_0$ and $\tau_0$ but we do not make it explicit to lighten the notation). Our goal is to estimate the quantity $I_\lambda(\Gamma_\varepsilon)$ (recall the definition of $I_\lambda$ in \eqref{defIlambda}).

In view of applying Lemma \ref{Lemme2}, we denote by $\Gamma_0$ the Hausdorff limit of $\Gamma_\varepsilon$, which surely exists up to extracting a subsequence. Let  $\lambda>0$ be a small enough parameter, and $\varepsilon_0>0$ be given by Lemma \ref{Lemme2} applied with $L<\Hh^1(\Gamma_0)$, in such a way that
\begin{eqnarray}
L\leq CI_\lambda(\Gamma_\varepsilon) \quad \forall \varepsilon<\varepsilon_0. \label{tricky}
\end{eqnarray}

Let now $\nu \in S^1$ be an arbitrary direction: for any $t \in \R$ we denote by $L_t$ the affine line $\R \nu + t \nu^{\bot}$. Since $\Gamma_\varepsilon$ is a finite union of curves of finite length, we know that
$$\Hh^0(L_t \cap \Gamma_\varepsilon)<+\infty \; , \quad \text{ for a.e. } t \in \R.$$
Let $G\subset \R$ be the set of such $t\in \R$ for which  $\Hh^0(L_t \cap \Gamma_\varepsilon)<+\infty$, and pick any $t \in G$.  Let $\{x_j\}_{j \in J}$ be the finite set of $\Gamma_\varepsilon \cap L_t$. We brutally identify $x_j$ with its coordinate on the line $L_t$ and we assume that they are labelled in increasing order, i.e. $x_j< x_{j+1}$. Next, we decompose the relative interior $Int((\Gamma_\varepsilon)_{\lambda,\nu}\cap L_t)$ as follows :

$$Int((\Gamma_\varepsilon)_{\lambda,\nu}\cap L_t) = \bigcup_{j \in J}I_j^- \cup I_j^+,$$
where
\[
I_j^+:=
\left\{
\begin{array}{ll}
[x_j, x_j+\lambda)  & \text{ if }  |x_{j+1}-x_{j}|\geq 2\lambda \\
$[$ x_j, \frac{x_j+x_{j+1}}{2} $]$& \text{otherwise}
\end{array}
\right.
\]

\[
I_j^-:=
\left\{
\begin{array}{ll}
(x_j-\lambda, x_j ]  & \text{ if }  |x_{j}-x_{j-1}|\geq 2\lambda \\
$($ \frac{x_j+x_{j-1}}{2} , x_j$]$& \text{ otherwise }
\end{array}
\right.
\]

Where by convention  $x_{-1}=-\infty$ and $x_{N+1}=+\infty$ so that $I_1^-$ and $I_N^+$ are well defined.  Notice also that  if $x_j$ lies too close to $\partial \Omega$, it could be that $I_j^+$ and $I_j^-$ goes a bit outside $\Omega$, but this will not be a problem in the sequel.
Indeed, the function $\varphi_\ve$ can be extended to the whole $\R^2$ by taking the value $1$ outside $\Omega$ (this is a consequence of $1-\varphi_\ve\in H^1_0(\Omega)$). From the definition of the functional $F_\varepsilon(v_\varepsilon,\varphi_\varepsilon, y_\varepsilon)$, extending with a constant value $1$ does not change the integral, since we have $\nabla\varphi_\ve=0$ and $(1-\varphi_\ve)^2=0$ outside $\Omega$. Hence, we can think that the functions are defined on a larger domain $\Omega'$ including $\overline\Omega$ in its interior and avoid caring about boundary issues.

Let $P(s):=s-s^2/2$ be the primitive of $s\mapsto (1-s)$ satisfying $P(0)=0$ and $P(1)=1/2$. Arguing as Modica and Mortola \cite{MM}, using the inequality $\frac{1}{4\varepsilon}a^2+\varepsilon b^2 \geq ab$, we infer that
\begin{eqnarray}
C&\geq& F_\varepsilon(v_\varepsilon,\varphi_\varepsilon, y_\varepsilon)\geq   \frac{1}{4\varepsilon}\int_{\Omega}(1-\varphi_\varepsilon)^2 dx + \varepsilon \int_{\Omega} \modvec \nabla \varphi_\varepsilon\modvec ^2 dx \notag \\
&\geq&  \int_\Omega (1-\varphi_\varepsilon) \modvec \nabla \varphi_\varepsilon \modvec  dx \geq  \int_\Omega \modvec \nabla (P(\varphi_\varepsilon))\modvec  dx. \notag
\end{eqnarray}
%Now we use the fact that $\Omega$ is a Lipchitz  domain, thus  in particular an extension domain for $W^{1,1}(\R^N)$ (see \cite{Jones}), to define $P(\varphi_\varepsilon)$ outside $\Omega$. For sake of simplicity we still denote this function by $P(\varphi_\varepsilon)$: it belongs to $W^{1,1}(\R^2)$, and we assume it to be compactly supported and satisfying
% $$\int_{\R^2} \modvec \nabla (P(\varphi_\varepsilon))\modvec  dx \leq C,$$
% where $C$ here depends also on the norm of the extension operator associated with $\Omega$.   A fortiori, we have that
Hence, we can go on with
\begin{eqnarray}
C &\geq &  \int_{\Omega'} |\nabla_{\nu} (P(\varphi_\varepsilon))| dx = \int_{\R} Var(f_t,L_t\cap\Omega'), \label{inequu}
\end{eqnarray}
where $f_t:=P(\varphi_\varepsilon)|_{L_t}$. On the other hand  applying Lemma \ref{variation} we can write
\begin{eqnarray}
Var(f_t, L_t\cap \Omega)\geq Var(f_t,  Int((\Gamma_\varepsilon)_{\lambda,\nu}\cap L_t)) \geq \sum_{j \in J}m_j^++m_j^- -2f_t(x_j), \label{estimationVar}
\end{eqnarray}
where $m_j^\pm$ denotes the average of $f_t$ on $I^\pm_j$. Since $\Hh^1(I_j^\pm)\leq \lambda$ for all $j$, we deduce that
\begin{eqnarray}
Var(f_t, L_t\cap \Omega)\geq \frac{1}{\lambda}\int_{(\Gamma_\varepsilon)_{\lambda,\nu}\cap L_t} f_t(s) ds - 2 \sum_{j \in J}f_t(x_j). \notag
\end{eqnarray}
Integrating over $t\in \R$, applying Fubini's Theorem   and using \eqref{inequu} it comes
\begin{eqnarray}
\frac{1}{\lambda}\int_{(\Gamma_\varepsilon)_{\lambda,\nu}}  P(\varphi_\varepsilon(x))dx - 2 \int_{\R}\sum_{j \in J}f_t(x_j) dt \leq C \label{inequneige}
\end{eqnarray}
Now we estimate the second term in the left hand side of \eqref{inequneige}. The co-area formula  (see for instance the equality (2.72) page 101 of \cite{afp}) applied on the $1$-rectifiable set  $\Gamma_\varepsilon$ provides
$$ \int_{\R}\sum_{j \in J}f_t(x_j) dt =   \int_{\Gamma_\varepsilon} P(\varphi_\varepsilon(x))C_\nu(x)d\Hh^1(x),$$
where $C_\nu(x)$ denotes the one dimensional co-area factor associated with the orthogonal projection $x \mapsto \langle x,\nu^{\bot}\rangle$. Since the latter mapping is $1$-Lipschitz, it is easy to verify that $|C_\nu(x)|\leq 1$ yielding
\begin{eqnarray}
 \int_{\R}\sum_{j \in J}f_t(x_j) dt &\leq&    \int_{\Gamma_\varepsilon} P(\varphi_\varepsilon(x))d\Hh^1(x) \notag \\
 &\leq&  \sum_{1\leq i\leq N}\int_{\Gamma_i^\varepsilon} P(\varphi_\varepsilon(x))d\Hh^1(x) \notag \\
 &\leq &  \sum_{1\leq i\leq N}\int_{\Gamma_i^\varepsilon} \varphi_\varepsilon(x)d\Hh^1(x) \label{onsort} \\
 &\leq & N\delta_0. \label{onsort2}
\end{eqnarray}
For \eqref{onsort} we have used $P(\varphi_\varepsilon)=\varphi_\varepsilon-\varphi_\varepsilon^2/2\leq \varphi_\varepsilon$, and for \eqref{onsort2} we have used \eqref{courbespetites}.

Returning now to \eqref{inequneige}, we have proved that
\begin{eqnarray}
\frac{1}{\lambda}\int_{(\Gamma_\varepsilon)_{\lambda,\nu}}  P(\varphi_\varepsilon(x))dx  \leq C + 2N\delta_0. \label{jolibound0}
\end{eqnarray}
Let us emphasis that $C$ is a uniform  constant, but $N$ depends  on $\tau_0$ (and is independent of $\nu$ and $\delta_0$). From \eqref{jolibound0}  we get
\begin{eqnarray}
\frac{1}{2\lambda} \mathscr{(L}^2((\Gamma_\varepsilon)_{\lambda,\nu})  =\frac{1}{\lambda}\int_{(\Gamma_\varepsilon)_{\lambda,\nu}}  P(1)dx  \leq C + 2N\delta_0 + \frac{1}{\lambda}\int_{\Omega'}|P(1)-P(\varphi_\varepsilon)|, \notag
\end{eqnarray}
and finally, taking the average over $\nu \in \mathbb{S}^1$, we get
\begin{eqnarray}
\frac{1}{2}I_\lambda(\Gamma_\varepsilon )= \frac{1}{4\pi\lambda} \int_{\mathbb{S}^1}\mathscr{(L}^2((\Gamma_\varepsilon)_{\lambda,\nu}) d\nu  \leq C + 2N\delta_0 + \frac{1}{\lambda}\int_{\Omega'}|P(1)-P(\varphi_\varepsilon)|dx. \notag
\end{eqnarray}
Then, due to \eqref{tricky}, it follows that, after fixing $\tau_0$ and $\delta_0$, getting some set $\Gamma_\ve$ and a limit set $\Gamma_0$, the inequality
\begin{eqnarray}
L \leq 2C + 4N\delta_0+ \frac{2}{\lambda}\int_{\Omega'}|P(1)-P(\varphi_\varepsilon)|dx \label{yuyu}
\end{eqnarray}
holds for all $L<\Hh^1(\Gamma_0)$, for the value of $\lambda$ that we have fixed, and $\varepsilon<\varepsilon_0$. We now let $\varepsilon \to 0$.  Since  $P(\varphi_\varepsilon)$ is uniformly bounded in $L^\infty(\Omega)$, and $\varphi_\varepsilon\to 1$ a.e. in $\Omega$, we obtain $P(\varphi_\varepsilon)\to P(1)$ strongly in $L^1(\Omega)$ (and $L^1(\Omega')$), which yields
\begin{eqnarray}
L \leq 2C + 4N\delta_0. \label{yuyu20}
\end{eqnarray}
$L$ being an arbitrary number smaller than $\Hh^1(\Gamma_0)$, we get
\begin{eqnarray}
\Hh^1(\Gamma_0) \leq 2C + 4N\delta_0. \label{yuyu20sans L}
\end{eqnarray}
Recall that $\Gamma_0$ is a connected set containing the $\tau_0$-dense set of points $\{z_i \}\subseteq K$, of total number $N$ depending only on $\tau_0$. Recall also that $\Gamma_0$ is obtained as the Hausdorff limit of union of curves $\Gamma_i^\varepsilon$, that are defined upon the parameter $\delta_0$. Therefore, $\Gamma_0$ depends a priori on $\delta_0$ as well. This is why we define, up to a subsequence of $\delta_0\to 0$, the Hausdorff limit of $\Gamma_0(\delta_0)$ that we denote by $\Gamma_{00}$. The set $\Gamma_{00}$ is still a connected set containing the $\tau_0$-dense set of points $\{z_i \}\subset K$. By passing to the liminf in \eqref{yuyu20} and by use of Golab Theorem we get
\begin{eqnarray}
\Hh^1(\Gamma_{00}) \leq\liminf_{\delta_0\to 0} \Hh^1(\Gamma_0) \leq 2C . \label{yuyu3}
\end{eqnarray}
But now letting $\tau_0\to 0$ we get, through a suitable subsequence, a Hausdorff limit set $\Gamma_{000}$ which is still connected and surely contains $K$ because $\Gamma_{00}$ contained a $\tau_0-$network of $K$. By use of Golab Theorem again we get
\begin{eqnarray}
\Hh^1(K)\leq \Hh^1(\Gamma_{000}) \leq\liminf_{\tau_0\to 0} \Hh^1(\Gamma_{00}) \leq 2C . \label{yuyu3}
\end{eqnarray}
This implies \eqref{limiinf0} and finishes the proof of the first step.

\vspace{0.5cm}
\noindent \emph{Step 2. More precise lower bound. } Now that we know that $K$ is rectifiable, we will improve the lower bound. Namely, we shall now prove $(d)$.

For this purpose we consider the following family of measures supported on $\overline{\Omega}$,
$$m_\varepsilon = \left(\frac{1}{4\varepsilon}(1-\varphi_\varepsilon)^2  + \varepsilon \modvec \nabla \varphi_\varepsilon\modvec ^2 \right)\mathscr{L}^2|_{\Omega},$$
that we assume to be weakly-$*$ converging to some measure $m$ supported on $\overline{\Omega}$ (this is not restrictive up to extracting a subsequence, thanks to the bound in $(iv)$).

Applying Lemma \ref{rectifiability} to the set $K$, we know that
$\Hh^1$-a.e. point $x\in K$ admits a tangent line. Let $T_x$ be the tangent
line to $x$. We assume  without loss of generality that $x=0$ and
$T_x=\R e_1$.  We denote by $\pi$ the orthogonal projection onto the one-dimensional vector space $\R e_1$. Let $0< \lambda <1$ be fixed (very close to $1$).
Lemma \ref{rectifiability} provides that for some $r_0>0$ and for all
$r\leq r_0$,  (that we suppose small enough so that
$B(x,r)\subseteq \Omega$)
\begin{eqnarray}
\pi(K\cap B(x, r)) \supseteq  [-\lambda r, \lambda r]. \label{measproj0}
\end{eqnarray}
%But this obviously implies that
%$$\pi(K\cap B%(x,r)\cap \pi^{-1}())=  [-\lambda r, \lambda r],$$
%thus changing $r$ into $\lambda r$ if necessarily, we can assume that
%$$\pi(K\cap C(x,  r))=  [- r,  r]$$%for all $r\leq r_0$.

Then  we consider the rectangle
$$C_{\lambda}(r):= [-\lambda r, \lambda r ]\times [-hr,hr],$$
with $h= \sqrt{1-\lambda^2}$, so that $C_\lambda(r)\subset B(x,r)$ (see Picture 2 below). We want to estimate $m_\varepsilon(C_{\lambda}(r))$, for $r$ small.

Let $\beta>0$ be a very small parameter satisfying
\begin{eqnarray}
\beta \leq  10^{-3}h. \label{defbeta}
\end{eqnarray}
Up to taking a smaller $r_0$ we may also suppose that for all $r\leq r_0$,  $K \cap C_{\lambda}(r)\subset W(\beta,r)$, where $W(\beta,r)$ is a small strip near the tangent of relative width $\beta$, namely
\begin{eqnarray}
W(\beta,r):= \{y \in C_{\lambda}(r) ; d(y,T_x)\leq  r\beta\}. \label{defW}
\end{eqnarray}
Let us define $\delta_\ve:=||d_{\varphi_\ve}(\cdot,x_\ve)-d||_{L^\infty}$, which is a sequence converging to $0$ as $\ve\to 0$, and consider the sets $K_{\ve,\delta_\ve}$, which converge in the Hausdorff topology to $K$. Moreover, we have $K\subset K_{\ve,\delta_\ve}$. We can also define some $\varepsilon_0>0$ small enough so that :
\begin{eqnarray}
K_{\varepsilon,\delta_\ve}\subseteq W(2\beta,r)\quad \forall \varepsilon \leq \varepsilon_0;\label{confinement}
\end{eqnarray}
Finally, we can also guarantee that
\begin{eqnarray}
x_\varepsilon \in K_{\varepsilon,\delta_\varepsilon} \setminus C_{\lambda}(r) \quad \forall \varepsilon \leq \varepsilon_0 \text{ and } \forall r\leq r_0. \label{haha}
\end{eqnarray}

Under those conditions we are sure that $K\cap B(x,r)\cap \pi^{-1}([-\lambda r,\lambda r])\subset C_{\lambda}(r)$ thus it follows from \eqref{measproj0} that for all $t\in  [- \lambda r,   \lambda r]$ we can find a point $z_t$  that belongs to $\pi^{-1}(t)\cap K$ (and hence also to $\pi^{-1}(t)\cap K_{\ve,\delta_\ve}$). Since $d_{\varphi_\ve}(z_t)<\delta_\ve$, thanks to \eqref{haha}, for every $z_t$ there exists a curve $\Gamma_{\varepsilon}(t)$ connecting $z_t$ to $x_\varepsilon$ and such that
$$\int_{\Gamma_{\varepsilon}(t)}\varphi_{\varepsilon}(s)ds\leq \delta_\ve.$$
Because of \eqref{confinement}, and the fact that $x_\varepsilon$ lies outside $C_{\lambda}(r)$, the curve $\Gamma_\varepsilon(t)$ must first exit $C_{\lambda}(r)$ on either the left or the right side of $C(x,r)$. More precisely, denoting by $C^+$ and $C^-$ the two connected components of $\partial C_{\lambda}(r)\cap W(\beta,r)$,
$$C^\pm:= \{y\in \partial C_\lambda(r); d(x,(\pm  r,0))\leq \beta r\},$$
we must have that
$$\Gamma_\varepsilon(t)\cap C^+ \not = \emptyset\text{ or }\Gamma_\varepsilon(t)\cap C^- \not = \emptyset$$
as in the picture below.
\vspace{1cm}
\begin{center}
% Generated with LaTeXDraw 2.0.8
% Fri Mar 15 15:28:16 CET 2013
 %\usepackage[usenames,dvipsnames]{pstricks}
 %\usepackage{epsfig}
 %\usepackage{pst-grad} % For gradients
 %\usepackage{pst-plot} % For axes
\scalebox{1} % Change this value to rescale the drawing.
{
\begin{pspicture}(0,-2.89)(12.8,2.89)
\definecolor{color1b}{rgb}{0.8,0.8,1.0}
\definecolor{color124b}{rgb}{0.2,0.6,1.0}
\definecolor{color155b}{rgb}{0.8,0.8,0.8}
\psellipse[linewidth=0.04,dimen=outer](4.97,0.0)(2.89,2.89)
\psframe[linewidth=0.04,dimen=outer,fillstyle=solid,fillcolor=color1b](7.76,0.65)(2.18,-0.55)
\usefont{T1}{ptm}{m}{n}
\rput(10.37,0.315){$h=\sqrt{2\lambda(1-\lambda)}$}
\usefont{T1}{ptm}{m}{n}
\rput(4.86,0.975){$2\lambda r$}
\usefont{T1}{ptm}{m}{n}
\rput(5.19,-0.805){$C_{\lambda}(r)$}
\usefont{T1}{ptm}{m}{n}
\rput(0.91,0.135){$T_x$}
\usefont{T1}{ptm}{m}{n}
\rput(7.56,2.475){$B(x,r)$}
\usefont{T1}{ptm}{m}{n}
\rput(9.73,1.675){$K$}
\psframe[linewidth=0.018,dimen=outer,fillstyle=solid,fillcolor=color124b](7.74,0.23)(2.2,-0.13)
\psline[linewidth=0.04cm,fillcolor=color155b](8.16,0.07)(1.52,0.05)
\psline[linewidth=0.03cm,fillcolor=color155b,arrowsize=0.05291667cm 2.0,arrowlength=1.4,arrowinset=0.4]{<-}(2.86,-0.15)(1.56,-1.93)
\usefont{T1}{ptm}{m}{n}
\rput(1.59,-2.145){$W(\beta,r)$}
\psline[linewidth=0.03cm,fillcolor=color155b,arrowsize=0.05291667cm 2.0,arrowlength=1.4,arrowinset=0.4]{<-}(2.22,-0.05)(1.12,-0.97)
\usefont{T1}{ptm}{m}{n}
\rput(0.92,-1.205){$C^-$}
\pscustom[linewidth=0.03]
{
\newpath
\moveto(9.48,1.47)
\lineto(9.23,1.42)
\curveto(9.105,1.395)(8.92,1.36)(8.86,1.35)
\curveto(8.8,1.34)(8.715,1.31)(8.69,1.29)
\curveto(8.665,1.27)(8.615,1.225)(8.59,1.2)
\curveto(8.565,1.175)(8.525,1.13)(8.51,1.11)
\curveto(8.495,1.09)(8.47,1.04)(8.46,1.01)
\curveto(8.45,0.98)(8.405,0.895)(8.37,0.84)
\curveto(8.335,0.785)(8.28,0.71)(8.26,0.69)
\curveto(8.24,0.67)(8.195,0.615)(8.17,0.58)
\curveto(8.145,0.545)(8.1,0.49)(8.04,0.43)
}
\pscustom[linewidth=0.03]
{
\newpath
\moveto(8.36,0.81)
\lineto(8.74,0.81)
\curveto(8.93,0.81)(9.185,0.81)(9.25,0.81)
\curveto(9.315,0.81)(9.38,0.8)(9.38,0.77)
}
\pscustom[linewidth=0.03]
{
\newpath
\moveto(8.12,0.49)
\lineto(8.12,0.43)
\curveto(8.12,0.4)(8.115,0.35)(8.11,0.33)
\curveto(8.105,0.31)(8.075,0.28)(8.05,0.27)
\curveto(8.025,0.26)(7.98,0.24)(7.96,0.23)
\curveto(7.94,0.22)(7.9,0.205)(7.88,0.2)
\curveto(7.86,0.195)(7.815,0.185)(7.79,0.18)
\curveto(7.765,0.175)(7.715,0.165)(7.69,0.16)
\curveto(7.665,0.155)(7.61,0.145)(7.58,0.14)
\curveto(7.55,0.135)(7.49,0.13)(7.46,0.13)
\curveto(7.43,0.13)(7.375,0.13)(7.35,0.13)
\curveto(7.325,0.13)(7.27,0.13)(7.24,0.13)
\curveto(7.21,0.13)(7.155,0.125)(7.13,0.12)
\curveto(7.105,0.115)(7.04,0.11)(7.0,0.11)
\curveto(6.96,0.11)(6.895,0.11)(6.87,0.11)
\curveto(6.845,0.11)(6.765,0.11)(6.71,0.11)
\curveto(6.655,0.11)(6.545,0.105)(6.49,0.1)
\curveto(6.435,0.095)(6.35,0.085)(6.32,0.08)
\curveto(6.29,0.075)(6.22,0.05)(6.18,0.03)
\curveto(6.14,0.01)(6.05,-0.02)(6.0,-0.03)
\curveto(5.95,-0.04)(5.875,-0.055)(5.85,-0.06)
\curveto(5.825,-0.065)(5.775,-0.07)(5.75,-0.07)
\curveto(5.725,-0.07)(5.68,-0.06)(5.66,-0.05)
\curveto(5.64,-0.04)(5.585,-0.02)(5.55,-0.01)
\curveto(5.515,0.0)(5.45,0.015)(5.42,0.02)
\curveto(5.39,0.025)(5.34,0.035)(5.32,0.04)
\curveto(5.3,0.045)(5.275,0.05)(5.26,0.05)
}
\pscustom[linewidth=0.03]
{
\newpath
\moveto(4.04,0.05)
\lineto(3.96,0.05)
\curveto(3.92,0.05)(3.85,0.055)(3.82,0.06)
\curveto(3.79,0.065)(3.74,0.075)(3.72,0.08)
\curveto(3.7,0.085)(3.66,0.095)(3.64,0.1)
\curveto(3.62,0.105)(3.575,0.115)(3.55,0.12)
\curveto(3.525,0.125)(3.47,0.13)(3.44,0.13)
\curveto(3.41,0.13)(3.355,0.13)(3.33,0.13)
\curveto(3.305,0.13)(3.26,0.135)(3.24,0.14)
\curveto(3.22,0.145)(3.175,0.15)(3.15,0.15)
\curveto(3.125,0.15)(3.075,0.15)(3.05,0.15)
\curveto(3.025,0.15)(2.985,0.15)(2.94,0.15)
}
\pscustom[linewidth=0.03]
{
\newpath
\moveto(2.94,0.15)
\lineto(2.88,0.15)
\curveto(2.85,0.15)(2.79,0.15)(2.76,0.15)
\curveto(2.73,0.15)(2.665,0.15)(2.63,0.15)
\curveto(2.595,0.15)(2.525,0.15)(2.49,0.15)
\curveto(2.455,0.15)(2.395,0.15)(2.37,0.15)
\curveto(2.345,0.15)(2.295,0.15)(2.27,0.15)
\curveto(2.245,0.15)(2.18,0.15)(2.14,0.15)
\curveto(2.1,0.15)(2.015,0.155)(1.97,0.16)
\curveto(1.925,0.165)(1.825,0.19)(1.77,0.21)
\curveto(1.715,0.23)(1.62,0.275)(1.58,0.3)
\curveto(1.54,0.325)(1.48,0.37)(1.46,0.39)
\curveto(1.44,0.41)(1.415,0.43)(1.4,0.43)
}
\pscustom[linewidth=0.03]
{
\newpath
\moveto(2.8,0.15)
\lineto(2.76,0.13)
\curveto(2.74,0.12)(2.705,0.09)(2.69,0.07)
\curveto(2.675,0.05)(2.645,0.015)(2.63,0.0)
\curveto(2.615,-0.015)(2.575,-0.035)(2.55,-0.04)
\curveto(2.525,-0.045)(2.475,-0.05)(2.45,-0.05)
\curveto(2.425,-0.05)(2.38,-0.045)(2.36,-0.04)
\curveto(2.34,-0.035)(2.305,-0.02)(2.29,-0.01)
\curveto(2.275,0.0)(2.24,0.015)(2.22,0.02)
\curveto(2.2,0.025)(2.16,0.02)(2.14,0.01)
\curveto(2.12,0.0)(2.085,-0.02)(2.07,-0.03)
\curveto(2.055,-0.04)(2.015,-0.06)(1.99,-0.07)
\curveto(1.965,-0.08)(1.92,-0.1)(1.9,-0.11)
\curveto(1.88,-0.12)(1.835,-0.145)(1.81,-0.16)
\curveto(1.785,-0.175)(1.74,-0.2)(1.72,-0.21)
\curveto(1.7,-0.22)(1.655,-0.25)(1.63,-0.27)
\curveto(1.605,-0.29)(1.53,-0.34)(1.48,-0.37)
\curveto(1.43,-0.4)(1.355,-0.445)(1.33,-0.46)
\curveto(1.305,-0.475)(1.265,-0.505)(1.22,-0.55)
}
\pscustom[linewidth=0.03]
{
\newpath
\moveto(1.24,-0.53)
\lineto(1.1,-0.53)
\curveto(1.03,-0.53)(0.895,-0.53)(0.83,-0.53)
\curveto(0.765,-0.53)(0.605,-0.53)(0.51,-0.53)
\curveto(0.415,-0.53)(0.275,-0.52)(0.23,-0.51)
\curveto(0.185,-0.5)(0.12,-0.475)(0.1,-0.46)
\curveto(0.08,-0.445)(0.04,-0.41)(0.02,-0.39)
\curveto(0.0,-0.37)(-0.015,-0.34)(0.0,-0.31)
}
\psline[linewidth=0.03cm,fillcolor=color155b,arrowsize=0.05291667cm 2.0,arrowlength=1.4,arrowinset=0.4]{<->}(7.74,0.71)(2.18,0.71)
\psline[linewidth=0.03cm,fillcolor=color155b,arrowsize=0.05291667cm 2.0,arrowlength=1.4,arrowinset=0.4]{<->}(8.46,0.61)(8.46,0.05)
\psline[linewidth=0.03cm,fillcolor=color155b,arrowsize=0.05291667cm 2.0,arrowlength=1.4,arrowinset=0.4]{<-}(7.76,-0.07)(8.72,-0.89)
\usefont{T1}{ptm}{m}{n}
\rput(9.36,-0.885){$C^+$}
\end{pspicture} 
}

\vspace{1cm}
\end{center}
Let us define
$$t_R:= \inf \{t \in [-r,r]; \Gamma_\varepsilon(t)\cap C^+ \not = \emptyset \}$$
$$t_L:= \sup \{t  \in [-r,r]; \Gamma_\varepsilon(t)\cap C^- \not = \emptyset \}.$$
Of course the two sets on which we take the lower and upper bounds are not empty because
$\Gamma_\varepsilon(\pm\lambda r)\cap C^\pm = \{(\pm  r,0)\} \not = \emptyset$. Also notice that we necessarily have $t_R\leq t_L$. Indeed, the opposite inequality holds, then for all the points $t\in ]t_L,t_R[$, the curve $\Gamma_\ve(t)$ would not meet neither $C^+$ not $C^-$, which is impossible.  Then, take $t'_L<t_L$ and $t'_R<t_R$ such that $|t'_R-t'_L|<\ve$ and define
$$\Gamma_\varepsilon:= \Gamma_\varepsilon(t'_L)\cup\Gamma_\varepsilon(t'_R).$$
The set $\Gamma_\varepsilon$ is not necessarily connected, but we have
\begin{eqnarray}
\pi(\Gamma_\varepsilon)\supset I_{\lambda,r,\ve}:= [- \lambda r,  \lambda r]\,\setminus\, ]t'_L,t'_R[. \label{projectionn}
\end{eqnarray}
For every $t\in  I_{\lambda,r,\ve}$ we denote by $g_t$ a point in $\Gamma_\ve\cap\pi^{-1}(t)$.

Let us  now estimate
\begin{eqnarray}
m_\varepsilon(C_\lambda ( r))&=&    \frac{1}{4\varepsilon}\int_{C_{\lambda}(r)}(1-\varphi_\varepsilon)^2 dx + \varepsilon \int_{C_{\lambda}( r)} \modvec \nabla \varphi_\varepsilon\modvec ^2 dx \notag \\
&\geq&  \int_{C_{\lambda}( r)} (1-\varphi_\varepsilon) \modvec \nabla \varphi_\varepsilon \modvec  dx \geq  \int_{C_{\lambda}( r)} \modvec \nabla (P(\varphi_\varepsilon))\modvec  dx \notag \\
&\geq &  \int_{C_{\lambda}( r)} \Big|\frac{\partial}{\partial x_2}(P(\varphi_\varepsilon))\Big| dx \geq \int_{-\lambda r}^{ \lambda r} Var(f_t,[-hr,hr])dt \notag
\end{eqnarray}
where  $f_t:=P(\varphi_\varepsilon)|_{L_t}$ with $L_t:=t+\R e_2$. On the other hand  applying Lemma \ref{variation}, for every $t\in  I_{\lambda,r,\ve}$, we can write
\begin{eqnarray}
Var(f_t, I_t ) \geq m^+_t+m^-_t -2f_t(g_t), \label{estimationVar}
\end{eqnarray}
where $m^+$ denotes the average of $f_t$ on $[g_t,hr ]$ and $m^-$ denotes the average of $f_t$ on $[-hr, g_t]$ (here we identified $g_t$ on the line $L_t$ with its coordinate on the second axis). Observe that the length of each of those two intervals lies between $hr - \beta r$ and $hr + \beta r$ , which is positive thanks to \eqref{defbeta}.
Next,
\begin{eqnarray}
Var(f_t, [-hr,hr] ) \geq \frac{1}{r(h+\beta)}\int_{I_t} f_t(s) ds - 2 f_t(g_t). \notag
\end{eqnarray}

%Notice that
%$$\mathscr{L}^2(C_\lambda(r))=\int_{-\lambda r}^{\lambda r}\int_{I_t}ds dt$$

Integrating over $t\in I_{\lambda,r,\ve}= [- \lambda r, \lambda r]\setminus [t'_L,t'_R]$ and applying Fubini's Theorem  it comes
\begin{eqnarray}
m_\varepsilon(C_{\lambda}(r))  &\geq  &\frac{1}{r(h+\beta)}\int_{I_{\lambda,r,\ve}\times [-hr,hr]}   P(\varphi_\varepsilon(x))dx - 2 \int_{I_{\lambda,r,\ve}}f_t(g_t) dt \label{inequneige00U}\\
&\geq & \frac{1}{r(h+\beta)}\int_{C_\lambda(r)}   P(\varphi_\varepsilon(x))dx -\frac{\ve h}{2(h+\beta)}- 2 \int_{I_{\lambda,r,\ve}}f_t(g_t) dt
\end{eqnarray}
To estimate the last term in the left hand side of \eqref{inequneige00U} we use the same argument as for \eqref{inequneige} relying on   the co-area formula  to  say that
\begin{eqnarray}
 \int_{I_{\lambda,r,\ve}}f_t(g_t) dt &\leq&    \int_{\Gamma_\varepsilon} P(\varphi_\varepsilon(x))d\Hh^1(x) \notag \\
 &\leq&  \int_{\Gamma_\varepsilon(t'_L)} P(\varphi_\varepsilon(x))d\Hh^1(x) +\int_{\Gamma_\varepsilon(t'_R)} P(\varphi_\varepsilon(x))d\Hh^1(x)\notag \\
 &\leq &  \int_{\Gamma_\varepsilon(t'_L)} \varphi_\varepsilon(x)d\Hh^1(x) +\int_{\Gamma_\varepsilon(t'_R)} \varphi_\varepsilon(x)d\Hh^1(x)\notag \label{onsort0} \\
 &\leq & 2\delta_\ve. \label{onsort20}
\end{eqnarray}
Recall that for \eqref{onsort0} we have used again $P(\varphi_\varepsilon)=\varphi_\varepsilon-\varphi_\varepsilon^2/2\leq \varphi_\varepsilon$.

Returning to \eqref{inequneige00U} we obtain that for all $r\leq r_0$ it holds
\begin{eqnarray}
m_\varepsilon(C_\lambda( r))  &\geq  &\frac{1}{r(h+\beta)}\int_{C_\lambda(r)}  P(\varphi_\varepsilon(x))dx -\frac\ve 2- 4\delta_\ve.\label{inequneige00}
\end{eqnarray}
Passing to the limsup in $\varepsilon\to 0$, using $\delta_\ve\to 0$, together with the facts that $\varphi_\varepsilon\to 1$ strongly in $L^1$, that $P(1)=1/2$, that $C_\lambda(r)$ is closed  and that $m_\varepsilon$  converges weakly-$*$ to $m$ we get
$$m(C_\lambda(r))\geq \limsup_{\varepsilon \to 0} m_\varepsilon(C(x, r))\geq \frac{1}{2 r(h +\beta)}\mathscr{L}^2(C_\lambda(r)).$$
Recalling that $C_\lambda(r)\subseteq B(x,r)$ we get
$$m(B(x, r))\geq \frac{1}{2r(h-\beta)} 4r^2h\lambda .$$
Finally, letting $\beta\to 0$ and $\lambda \to 1$ we get the density estimate
$$m(B(x,r))\geq 2r,$$
which leads to
$$\limsup_{r\to 0}\frac{m(B(x,r))}{2r}\geq 1,$$
and this holds for $\Hh^1$-a.e. $x\in K$. Applying   Proposition \eqref{densityLemma}, we find that
$$m\geq \Hh^1|_{K},$$
and we conclude that
\begin{eqnarray}
\Hh^1(K) \leq \liminf_{\varepsilon\to 0} \left( \frac{1}{4\varepsilon}\int_{\Omega}(1-\varphi_\varepsilon)^2 dx + \varepsilon\int_{\Omega}\modvec \nabla \varphi_\varepsilon \modvec ^2 dx \right).\notag\qedhere
\end{eqnarray}
 \end{proof}

%%%%%%%%%%%%%%%%%%%%
%\begin{remark} {\color{red} je serais pour \'eliminer cette remarque et obtenir le r\'esultat avec le bord (coeff 1 sur le bord, pas 1/2)} 
%Notice that in the above Lemma  did not consider the contribution of $K \cap \partial \Omega$. Actually if $\Omega$ was smooth enough, a more sophisticated  lower bound would be probably something like 
%\begin{eqnarray}
%\frac{1}{2}\Hh^1(K\cap \partial \Omega)+\Hh^1(K\cap \Omega) \leq \liminf_{\varepsilon\to 0} \left( \frac{1}{4\varepsilon}\int_{\Omega}(1-\varphi_\varepsilon)^2 dx + \varepsilon\int_{\Omega}\modvec \nabla \varphi_\varepsilon \modvec ^2 dx \right),\notag
%\end{eqnarray}
%but we decided to avoid this boundary effect as follows: in the approximative process we will enlarge the working domain $\Omega$ into a bigger one $U\supset \overline{\Omega}$ and artificially put the condition $\varphi\geq 1/2$ on a neighborhood of $\partial U$ in order to be sure that $K$ will never touch the boundary of the domain, now being $\partial U$.
%\end{remark}

%%%%%%%%%%%%%%%%%%%%%%%%%%%%%%%%%%%%%%%

\section{Approximation of the Steiner problem}
\label{sectionSteiner}

In this section we explain how to approximate the following Steiner problem in $\R^2$. Let $\mu$ be a probability measure on $\R^2$ with compact support; we denote by 
$$\mathcal{K}_\mu:= \{K \subset \R^2; \text{ compact, connected, and s.t. } \spt(\mu)\subset K\}.$$
We then investigate 
\begin{eqnarray}
\inf \{ \Hh^1(K)\;; K \in \mathcal{K}_\mu\}. \label{steiner}
\end{eqnarray}
Notice that here $\mu$ is only important through its support.
If the infimum in the above problem is finite, then the problem  admits a solution as a direct consequence of Blaschke and Golab's Theorem; notice that in general the minimal set is not unique. The most investigated case is the case of the so-called Steiner problem (see \cite{gilbpoll,ps2009} and the references therein), where we consider a finite number of points finite set of points $\{x_i\}=:D$ and we choose any measure $\mu$ such that $\spt(\mu)=D$, for instance $\mu=\frac{1}{\sharp D}\sum_{i} \delta_{x_i}$ .

 To approximate  \eqref{steiner} we introduce an open set $\Omega$ containing the convex hull of $D$ in its interior. This is just to avoid boundary problems. Indeed, it is easy to verify that a minimizer for \eqref{steiner} will always stay inside $\Omega$ otherwise its projection onto the convex hull would make a better competitor (such a projection is indeed a $1$-Lipschitz mapping). Therefore \eqref{steiner} is equivalent to the problem
 \begin{eqnarray}
\min \{ \Hh^1(K)\;; K \in \mathcal{K}_\mu\text{ and } K \subset \Omega\}. \label{steiner2}
\end{eqnarray}
 Now, to approximate \eqref{steiner2} we take any arbitrary chosen point $x_0\in D$.  Then recalling the definition of $d_\varphi$ in \eqref{defDphi} we introduce the family of functionals defined on $L^2(\Omega)$ by

\begin{eqnarray}
S_{\varepsilon}(\varphi)=
 \frac{1}{4\varepsilon}\int_{\Omega}(1-\varphi)^2 dx +\; \varepsilon\int_{\Omega}\modvec \nabla \varphi \modvec ^2 dx +\frac{1}{c_\varepsilon}\int_{\Omega}d_{\varphi}(x,x_0) d\mu(x), \label{Fepsilon}
\end{eqnarray}

if $\varphi \in W^{1,2}(\Omega)\cap C^0(\overline{\Omega})$ satisfies $0\leq \varphi\leq 1$, $\varphi=1$ on $\partial \Omega $, and $S_{\varepsilon}(\varphi)=+\infty$ otherwise.

\begin{defn} \label{approximateMin}We say that $\varphi_\varepsilon$ is a quasi-minimizing sequence for $S_\varepsilon$ if 
$$S_{\varepsilon}(\varphi_\varepsilon)-\inf_{\varphi}S_\varepsilon(\varphi)\underset{\varepsilon\to 0}\longrightarrow 0.$$
\end{defn}

Our approximation result is as follows.

\begin{thm} \label{SteinerTheorem} For all $\varepsilon>0$ let  $\varphi_{\varepsilon}$ be a quasi-minimizing sequence of $S_{\varepsilon}$, where  $c_\varepsilon\to 0$. Consider the sequence of functions $d_{\varphi_{\ve_n}}$ converging uniformly to a certain function $d$. Then the set $K:=\{d=0\}$ is compact and connected and is a solution to Problem \eqref{steiner}.
\end{thm}

\begin{remark}Notice that the assumption of  $d_{\varphi_\ve}$ converging to a function $d$ is not restrictive since they are all $1$-Lipschitz functions thus always converge uniformly up to a subsequence. 
\end{remark}

\begin{proof} Defining 
$$\tau_n:=\sup_{k\geq n} \modvec d_{\varphi_{\varepsilon_k}}(x,x_0) - d(x)\modvec _{\infty},$$
it is easy to see that 
$$K_n:=\{x \in \overline{\Omega} ; d_{\varphi_{\varepsilon_n}}(x,x_0)\leq \tau_n\}$$
converges, for the Hausdorff distance, to the set $K=\{d(x)=0\}$.

Now let $K_0 \in \mathcal{K}_\mu$ be a minimizer for the Steiner Problem \eqref{steiner2}   in particular $K_0\subset \Omega$ because it is contained in the convex hull of the support of $\mu$, in other words  $K_0\cap \partial \Omega = \emptyset$.

Then, let  $a_\varepsilon$ and $b_\varepsilon$ be the same parameters as in Lemma \ref{limsup}, and let $\psi_\varepsilon$ be the family of function given by Lemma \ref{limsup}, with $k_\varepsilon=0$. For $\varepsilon$ small enough we have $\psi_{\varepsilon}=1$ on $\partial \Omega$.     Since $\varphi_\varepsilon$ are quasi-minimizers of $S_\varepsilon$, it follows that  for all $n\geq 0$, it holds $S_{\varepsilon_n}(\varphi_{\varepsilon_n})\leq S_{\varepsilon_n}(\psi_{\varepsilon_n})+o
(1)$, thus taking the liminf, and noticing that $\frac{1}{c_\varepsilon}\int_{\Omega}d_{\psi_\varepsilon}(x,x_0) d\mu(x)=0$ for all $\varepsilon$ we infer that
\begin{eqnarray}
 \liminf_{n}S_{\varepsilon_n}(\varphi_{\varepsilon_n})\leq \limsup_{n}S_{\varepsilon_n}(\psi_{\varepsilon_n})\leq \Hh^1(K_0). \label{celleduhaut}
 \end{eqnarray}
In particular $S_{\varepsilon_n}(\varphi_{\varepsilon_n})\leq C$ so that Lemma \ref{mainlemma} applies, thus we obtain that $K \in \mathcal{K}_\mu$ and that
$$\Hh^1(K)\leq \liminf_{n\to +\infty} S_{\varepsilon_n}(\varphi_{\varepsilon_n}).$$
Gathering with \eqref{celleduhaut} we deduce that 
$$\Hh^1(K)\leq \Hh^1(K_0)$$
which proves that $K$ is a minimizer.
\end{proof}

\begin{remark} We refer the reader to Section \ref{existence issues} in order to check the technical tricks to guarantee existence for the approximating problems $\min S_\ve$.
\end{remark}

%%%%%%%%%%%%%%%%%%%%%%%%%%%%%%%%%%%%%%%%
\section{Approximation of the average distance and compliance problems}
\label{Section-approx-average}
\subsection{The dual problem}

Our strategy is first to change the problem into a simpler
form via a duality argument. Indeed, if one writes the energy
$\int u_K f$ in terms of a minimization problem, one finds that
Problem \eqref{pb20} has a min-max form. By duality, we will
turn this problem into a min-min.
Let us denote by
$$\mathcal{K}:= \{K \subset \overline{\Omega}; \text{ closed and connected}\},$$
$$\mathcal{A}_q:=\{(K,v):K\in\mathcal{K},v \in L^q(\Omega)\;\int_{\Omega} v\cdot\nabla \psi=\int_{\Omega}\psi f\;\mbox{for all }\psi\in C^1(\Omega),\psi=0\,\mbox{on } K\},$$
(the definition of $\mathcal{A}_p$ is the same as in the introduction, with the divergence and boundary conditions expressed in a weak sense).

\begin{prop}  For $p\in]1,+\infty[$ and $q=p'=\frac{p-1}{p}$, Problem \eqref{pb20} is equivalent to 
\begin{equation}
 \inf_{(K ,v) \in\mathcal{A}_p} \left\{ \frac{1}{q}\int_{\Omega}\modvec  v\modvec ^q dx + \Lambda \Hh^1(K)\right\} . \label{pb3}
\end{equation}
\end{prop}

\begin{proof}

Let $K \in \mathcal{K}$ and $u_K$ be a minimizer of \eqref{prob1bis}. Then the optimality condition yields
$$\frac{1-p}{p}\int_{\Omega}u_K f=\min_{u \in W^{1,p}_{K}(\Omega)}\int_{\Omega}\left(\frac{1}{p}\modvec \nabla u\modvec ^p -u f\right)dx.$$
Let $q$ be the conjugate exponent to $p$. We need to prove 
\begin{multline}\label{duality pq}
\min_{u \in W^{1,p}_{K}(\Omega\setminus K)}\int_{\Omega}\left(\frac{1}{p}\modvec \nabla u\modvec ^p -u f\right)dx\\
=\sup\left\{
 - \frac{1}{q}\int_{\Omega}|v|^q\ dx,
 \quad v\in L^q(\Omega,\R^2)\,:\,(K,v)\in\mathcal{A}_q\right\}.\end{multline}

This is quite classical, but we provide a proof from convex analysis  (the following approach   is inspired by the proof of Theorem 2 in \cite{chamb}). For  $\eta\in L^p(\Omega, \R^2)$, we set
\[
\Phi(\eta) := \inf_{u\in\uspace}\int_{\Omega} \left( \frac{1}{p}|\nabla u + \eta|^p - uf \right)dx.
\]
This functional $\Phi$ is convex in $\eta$, since it is obtained as the infimum over $u$ of a functional which is jointly convex in $(u,\eta)$. The argument from convex analysis that we use is the following : given a reflexive Banach space $E$, if $\Phi : E\rightarrow \overline{\R}$ is a convex function, lower semi-continuous, never taking the value $-\infty$, then $\Phi^{**}=\Phi$, where $\Phi^*$ denotes the convex conjugate (see for instance \cite{ekeland1976convex}). It is easy to check that our function $\Phi$  satisfies these extra conditions (l.s.c. and $\Phi>-\infty$).

We denote by $v\in L^q(\Omega,\R^2)$ the dual variable associated to $\eta$. Let us compute $\Phi^*(v)$ :
\begin{align*}
\Phi^*(v) & = \sup_{\eta\in L^p} \left[ \int_{\Omega}v \cdot \eta\ dx - \Phi(\eta)\right]\\
& = \sup_{\eta\in L^p} \left[ \int_{\Omega}v \cdot \eta\ dx - 
\inf_{u\in\uspace}
\left\lbrace
\int_{\Omega} \frac{1}{p}|\nabla u + \eta|^p dx - \int_{\Omega} fu \ dx
\right\rbrace
\right]\\
& = \sup_{\eta\in L^p} \left[ \int_{\Omega}v \cdot \eta\ dx + 
\sup_{u\in\uspace}
\left\lbrace
- \int_{\Omega} \frac{1}{p}|\nabla u + \eta|^p dx + \int_{\Omega} fu \ dx
\right\rbrace
\right]\\
& = \sup_{\eta\in L^p,\ u\in\uspace}
\left[
\int_{\Omega}v \cdot \eta\ dx
- \int_{\Omega} \frac{1}{p}|\nabla u + \eta|^p \ dx + \int_{\Omega} fu \ dx
\right]\\
& = \sup_{\eta\in L^p,\ u\in\uspace}
\left[
\int_{\Omega}\left\lbrace 
v \cdot (\nabla u + \eta)
- \frac{1}{p}|\nabla u + \eta|^p\right\rbrace \ dx 
+ \int_{\Omega} fu dx
- \int_{\Omega}v\cdot\nabla u\ dx
\right]
\end{align*}
Then we use the relation
\[
\sup_{\eta\in L^p}
\int_{\Omega}\left\lbrace 
v \cdot (\nabla u + \eta)
- \frac{1}{p}|\nabla u + \eta|^p\right\rbrace dx = \frac{1}{q}\int_{\Omega}|v|^q
\]
(the equality is achieved for $\eta=|v|^{q-2}v-\nabla u$), which yields :
\[
\Phi^*(v) = \frac{1}{q}\int_{\Omega}|v|^q\ dx + \sup_{u\in\uspace} \int_{\Omega} (fu
- v\cdot\nabla u)\ dx
\]
We then introduce the condition
\begin{equation}\label{condition}
\int_{\Omega} (fu -v\cdot\nabla u)\ dx = 0\quad \forall u\in\uspace
\end{equation}
(which says that $\mathrm{div}\ v=-f$   in $\Omega\setminus K$ and $v\cdot\nu =0$ on $\partial \Omega$, in a weak sense).
Since the above expression is linear in  $u$, at $v$ fixed, we see that the supremum in the above expression for $\Phi^*$ is either $0$ when $v$ satisfies  \eqref{condition}, or $+\infty$. This way,
\[
\Phi^*(v) = \left\lbrace  \begin{matrix}
\frac{1}{q}\int_{\Omega}|v|^q & \text{if}\ v\ \text{satisfies}\ \eqref{condition},\\
+\infty & \text{otherwise}.
\end{matrix}\right.
\] 
We take again the conjugate: for all $\eta\in L^p(\Omega,\R^2)$,
\begin{align*}
\Phi^{**}(\eta) & := \sup \left\lbrace
\int_{\Omega}v\cdot \eta\ dx - \Phi^*(v),
\quad v\in L^q(\Omega,\R^2)
\right\rbrace\\
& = \sup
\left\lbrace
\int_{\Omega}v\cdot \eta\ dx - \frac{1}{q}\int_{\Omega}|v|^q\ dx,
\quad v\in L^q(\Omega,\R^2),\ v\ \text{satisfies}\ \eqref{condition}
\right\rbrace
\end{align*}
We conclude by writing that $\Phi(0) = \Phi^{**}(0)$, in other words
\begin{multline*}
\inf\left\lbrace
\int_{\Omega} \left( \frac{1}{p}|\nabla u|^p - fu \right)dx,
\quad
u\in\uspace
\right\rbrace
=\\
\sup
\left\lbrace
 - \frac{1}{q}\int_{\Omega}|v|^q\ dx,
 \quad v\in L^q(\Omega,\R^2),\ v\ \text{satisfies}\ \eqref{condition} 
\right\rbrace.\qedhere
\end{multline*}
\end{proof}

\begin{remark}\label{sign div q>1}
From the previous proof, one notices that the optimal vector field $v$ satisfies $\dv v=\Delta_q u$, where $u$ is the solution of $-\Delta_q u=f$ with $u=0$ on $K$. The function $u$ is non-negative outside $K$, and hence its normal derivatives on $K$ are positives, and this gives the sign of $\dv v$ on $K$, which is a singular measure depending on these normal derivatives. 
\end{remark}

We also have a similar statement for the average distance problem.
Let us set 
$$\mathcal{A}_\infty:=\left\{(K,v):K\in\mathcal{K},\,v \in \mathcal M^d(\overline\Omega)   :\;\int_{\Omega}\! \nabla \psi\cdot dv=\!\int_{\Omega}\psi f\;\;\forall \psi\in C^1(\Omega),\psi=0\,\mbox{on } K\right\}.$$

Here $ \mathcal M^d(\overline\Omega) $ is the set of finite ($d-$dimensional) vector measures on $\overline\Omega$, endowed with the norm $||v||_{\mathcal M}:=|v|(\overline\Omega)=\sup\{\int \psi\cdot dv\,:\,\psi\in C^0(\overline\Omega;\R^d), |\psi|\leq 1\}$. Analogusly, we define $\mathcal M(\overline\Omega)$ as the set of finite signed measures on $\overline\Omega$.

\begin{prop}  Problem \eqref{prob1} is equivalent to 
\begin{equation}
 \inf_{(K ,v) \in \mathcal{A}_\infty} \left\{ \int_{\Omega}\modvec  v\modvec  dx + \Lambda \Hh^1(K)\right\} . \label{pb3.3}
\end{equation}
\end{prop}

\begin{proof} The proof for the case $p=\infty$ and the average distance problem could be obtained by adapting the previous proof for the compliance case, but would require some attention due to the fact that the spaces are non-reflexive. Hence, for the reader knowing some optimal transport techniques, we give a different approach.

First notice that 

$$\int_\Omega dist(x,K)f(x)dx=\min\{ W_1(f,\nu)\,:\,\spt(\nu)\subset K\},$$
where $W_1$ is the Wasserstein distance between probability measures, and we identify $f$ with a probability having $f$ as a density. It is easy to see that the optimal measure $\nu$ in the minimum above is given by $(\pi_K)_\#f$, where $\pi_K$ is the projection onto the set $K$ (well-defined a.e. and measurable). 

Next, we use Beckmann's interpretation of the distance $W_1$ (see for instance \cite{SanSMAI}), which gives
$$W_1(f,\nu)=\inf\{ ||v||_{\mathcal M}\;:\;-\dv v=f-\nu\},$$
where the divergence condition is to be intended in the sense 
$$\int \nabla\phi\cdot dv=\int \phi\,d(f-\nu),\quad\mbox{ for all }\phi\in C^1(\Omega),$$
(without compact support or boundary conditions on $\phi$, which means that $v$ also satisfies $v\cdot n_\Omega=0$). Hence we have
\begin{eqnarray*}
\int_\Omega dist(x,K)f(x)dx&=&\min_{\spt(\nu)\subset K}\min_{\dv v=f-\nu}||v||_{\mathcal M}\\
&=&\min\{||v||_{\mathcal M}\,:\,\spt(\dv v-f)\subset K\}.\qedhere
\end{eqnarray*}
\end{proof} 
\begin{remark}\label{sign div q=1}
From the previous proof, one notices that the optimal vector field $v$ satisfies $-\dv v=f-\nu$, where $\nu$ is a positive measure of the same mass as $f\geq 0$. In particular this gives the sign of the singular part of $\dv v$ and proves $|\dv v|(\Omega)=2\int_\Omega f(x)dx$.
\end{remark}
In the sequel we will assume that  $\Lambda=1$ to lighten the notation. Now we define $\mathcal{M}^{(q)}(\Omega)$ the space of fields $v \in L^q(\Omega,\R^2)$ whose divergence in the sense of distribution ${\rm div}\;v$ is a measure (for $q=1$, this becomes the space of vector measures such that the divergence is also a measure). For any signed measure $\mu$ on $\Omega$, we define the length of the ``connected envelope"  of its support by
$$CE(\mu):=St(\spt(\mu)),$$
where $St(A)$ is the length of the solution of the Steiner problem associated to the set $A$, namely, 
$$St(A):=\inf\{ \Hh^1(K) \;;\; K \subset \overline{\Omega}\text{ is closed, connected, and }A \subseteq K\}.$$
Then, for given $v\in \mathcal{M}^{(q)}(\Omega)$, we define  
$$\Hh^1_C(v):=CE(\dv v + f).$$
Notice that $\Hh^1_C(v)$ depends also on $f$ but we don't make it explicit.

\begin{remark} \label{thegoodremark}   If the infimum in the definition of $St(A)$ is finite, then it is actually a minimum. Indeed, let $K_n$ be a minimizing sequence such that $\Hh^1(K_n)<+\infty$ for $n$ large enough, by Golab's Theorem, up to extracting a subsequence we can assume that $K_n\to K_0$ for some closed and connected set $K_0$, and $\Hh^1(K_0)\leq \liminf\Hh^1(K_n)=St(A)$.
On the other hand,  the condition $A\subseteq K_n$, with $K_n$ converging to  $K_0$ in the Hausdorff topology, implies $A\subseteq K_0$. But then $K_0$ is admissible in the definition of $St(A)$ which implies that it is a minimizer.
\end{remark}

Next we define the functional that will arise as $\Gamma$-limit of our approximating functionals. If the following conditions are satisfied 
\begin{enumerate}
\item $v \in \mathcal{M}^d(\overline{\Omega})$ (for Problem \eqref{prob1}) or $v \in L^q(\Omega)$ (for Problem \eqref{pb20})
\item $\dv v \in \mathcal{M}(\overline{\Omega})$ 
\item  $y \in \spt(\dv v +f)$
\item $\spt(\dv v +f) \subset \overline{\Omega}$  
\item $\varphi =1$ a.e. on $ \Omega$ 
\end{enumerate}
then we set, for $q\geq 1$,
$$F_0(v,\varphi,y)=
\frac{1}{q}\int_{\Omega} \modvec v\modvec ^q dx + |\dv v|(\Omega)+  \Hh^1_C(v)$$
and $F_0(v,\varphi,y)=+\infty$ otherwise.

The next proposition says that the problem \eqref{prob1bis} is equivalent to the one of minimizing $F_0$.

\begin{prop} \label{propositionFilippo} If $f\geq 0$, $F_0$ has a minimizer, and finding it is equivalent to solving \eqref{pb3} or \eqref{pb3.3}.
\end{prop}

\begin{proof} 
To prove the existence of a minimizer for $F_0$ take a minimizing sequence $v_n$, and consider for each $n$ a compact and connected set $K_n$ such that $\Hh^1_C(v_n)=\Hh^1(K_n)$ and $\spt(\dv v+f)\subset K_n$. Up to a subsequence, we can assume $v_n\deb \bar v$ in $L^q(\Omega)$ (in $\M(\Omega)$ if $q=1$), $\dv v_n\destar \dv \bar v$ and $K_n\to K$ in the Hausdorff topology. Then we get $\spt( \dv \bar v+f)\subset K$, $\Hh^1_C(\bar v)\leq \Hh^1(K)\leq \liminf_n  \Hh^1(K_n)=\liminf_n \Hh^1_C(v_n)$. The semicontinuity of the other terms is immediate, and $\bar v$ is a minimizer. 

Then we want to prove that the minimizers of $F_0$ also minimize a simpler functional, which is given by
$$\tilde F_0(v):=\frac{1}{q}\int_{\Omega} \modvec v\modvec ^q dx +  \Hh^1_C(v)+2\int_\Omega f.$$
Indeed, the Neumann condition on the competitors $v$ implies $\int \dv v=0$, but $\dv v=-f$ on $\Omega\setminus K$. Then the mass of $\dv v$ is at least twice the integral of $f$ on $\Omega\setminus K$ which equals $\int_\Omega f$ since $K$ is Lebesgue-negligible. This shows that $F_0(v)\geq \tilde F_0(v)$ for every $v$. On the other hand, for any $v$ such that $\dv v\leq 0 $ on $K$ we have the  equality $F_0(v)= \tilde F_0(v)$  since the mass outside $K$ is exactly equal to that of the singular part on $K$. This is the case for any minimizer $\hat v$ of $\tilde F_0$ (see the Remarks \ref{sign div q>1} and \ref{sign div q>1}) and proves the equality of the two minimal value and the fact that the minimizers are the same.

Finally, the minimization of $F_1$ is obviously equivalent to that of $\frac{1}{q}\int_{\Omega} \modvec v\modvec ^q dx +  \Hh^1_C(v)$, since $2\int_\Omega f$ is a constant. This last problem is indeed a minimization in the pair $(K,v)$ with $\spt(\dv v+f)\subset K$, which is the same as minimizing in $\mathcal{A}_\infty$ and gives the problem \eqref{pb3}.
\end{proof}

We are now ready to  define the  family of functionals that will converge to $F_0$. We work on   $L^q(\Omega)\times L^2(\Omega) \times  \overline{\Omega}$. If the following conditions are satisfied
\begin{enumerate}
\item $v \in L^q(\Omega, \R^2)$ and $\dv v $ is a finite measure.
\item  $\varphi \in H^1(\Omega)\cap C^0(\Omega)$ 
\item $0\leq \varphi\leq 1$ 
\item $\varphi =1$ on $\partial \Omega$
\end{enumerate}
then
$$
F_{\varepsilon}(v,\varphi,y)=
\frac{1}{q}\int_{\Omega} \modvec v\modvec ^q dx + \frac{1}{4\varepsilon}\int_{\Omega}(1-\varphi)^2 dx + \varepsilon\int_{\Omega}\modvec \nabla \varphi \modvec ^2 dx $$
$$\hspace{1cm} +\frac{1}{\sqrt{\varepsilon}}\int_{\Omega}d_{\varphi}(x,y) d|\dv v +f|(x) + |\dv v|(\Omega) .
$$

Otherwise we set  $F_{\varepsilon}(v,\varphi,y)=+\infty$. 

%%%%%%%%%%%%%%%%%%%%

The rest of the paper is devoted to the proof of the following.

\begin{thm}  \label{Gammaconv}The family of functionals $F_{\varepsilon}$ $\Gamma$-converges to $F_0$ in the strong topology of $L^q\times L^2\times U$.
\end{thm}
As usual we split the $\Gamma$-convergence in two parts, corresponding to the $\Gamma$-liminf inequality and $\Gamma$-limsup inequality.

%%%%%%%%%%%%%%%%%%%%%%%%

\subsection{Proof of $\Gamma$-liminf}

\begin{thm} \label{gammaliminf}  Assume that $\Omega$ is   any open and bounded subset of  $\R^2$. 
Let $(v_\varepsilon, \varphi_\epsilon,y_\varepsilon)$ be a sequence converging weakly to some $(v,\varphi, y_0)$ in  $L^q(\Omega)\times L^2(\Omega) \times \overline{\Omega}$. Then $$F_0(v,\varphi,y_0)\leq \liminf_{\varepsilon\to 0} F_\varepsilon(v_\varepsilon, \varphi_\epsilon,y_\varepsilon).$$
\end{thm}

\begin{proof} Without loss of generality we may assume that
\begin{equation}
\liminf_\varepsilon F_\varepsilon(v_\varepsilon, \varphi_\epsilon,y_\varepsilon) <+\infty, \label{bound1}
\end{equation}
otherwise there is nothing to prove. We also assume that the liminf is a limit, achieved for some subsequence $\varepsilon_n \to 0$ that we still denote by $\varepsilon$ for simplicity. In particular during the proof, some further subsequences will be extracted, which does not affect the value of the limit, and we will still denote those sequences by $\varepsilon$.  As a consequence we have that
$$F_\varepsilon(v_\varepsilon, \varphi_\epsilon,y_\varepsilon) \leq C$$
for some constant $C$, and for $\varepsilon$ small enough, but forgetting the first terms we can also assume without loss of generality that it holds for all $\varepsilon$. In particular, we know that each of the terms of $F_\varepsilon$ is uniformly bounded and this implies, in what concerns the second term, that
\begin{equation}
 \varphi_\varepsilon \to 1 \text{ strongly in }L^{2}(\Omega), \label{convphi1}
\end{equation}
and for the last term, that
\begin{equation}
\text{there exists a signed measure } \mu \text{ such that } \dv v_\varepsilon   \destar \mu.
\end{equation}
Since $v_\varepsilon \to v$ weakly in $L^p$, by uniqueness of the limit in the distributional space $\mathcal{D}'(\Omega)$, we get that  $\dv v=\mu$ is a measure.

Next we focus on the distance functionals $d_{\varphi_\varepsilon}$. Since $\sup_\varepsilon \modvec \varphi_\varepsilon\modvec _\infty \leq 1$, the family of functions $x\mapsto d_{\varphi_\varepsilon}(x,y_\varepsilon)$ is equi-Lipschitz on $\overline{\Omega}$. Therefore, up to a subsequence, we may assume that $d_{\varphi_\varepsilon}(\cdot,y_\varepsilon)$ converges uniformly to a function $d(x)$ in $\overline{\Omega}$.

We are now in position to apply Lemma \ref{mainlemma} which says that 
\begin{enumerate}[{\rm ($a$)}]
\item  the compact set $K:=\{x \in \overline{\Omega} \; ;\; d(x)=0\}$ is  connected,
\item $y_0\in K$,
\item $\spt(\dv v+f)\subset K$,
\item
$
\Hh^1(K) \leq \liminf_{\varepsilon\to 0} \left( \frac{1}{4\varepsilon}\int_{\Omega}(1-\varphi_\varepsilon)^2 dx + \varepsilon\int_{\Omega}\modvec \nabla \varphi_\varepsilon \modvec ^2 dx \right).$
\end{enumerate}
 It follows that $\Hh^1(K)\geq \Hh^1_C(v)$ and 
$$F_0(v,\varphi,y)=
\frac{1}{q}\int_{U} \modvec v\modvec ^q dx + |\dv v|(U)+  \Hh^1_C(v).$$
Finally from the lower semicontinuity property with respect to the weak convergence we get
$$\frac{1}{q}\int_{\Omega} \modvec v\modvec ^q dx \leq \liminf_{\varepsilon}\frac{1}{q}\int_{\Omega} \modvec v_{\varepsilon}\modvec ^q dx $$
and 
$$  |\dv v| (\Omega)\leq \liminf_{\varepsilon}  |\dv v_\varepsilon|(\Omega),$$
which finishes the proof of the Theorem. 
\end{proof}

\subsection{Proof of $\Gamma$-limsup inequality}
%%%%%%%%%%%%%%%%%%%%%%%%%%

\begin{thm}  \label{gammalimsup}  Suppose that $\Omega$ is Lipschitz and star-shaped around $0\in\Omega$. Then, for any $(v,\varphi, y_0)$ in  $L^q(\Omega)\times L^2(\Omega) \times \Omega$  there exists a weakly converging sequence $(v_\varepsilon, \varphi_\epsilon,y_\varepsilon)\to (v,\varphi, y_0)$ such that
$$\limsup_{\varepsilon\to 0} F_\varepsilon(v_\varepsilon, \varphi_\epsilon,y_\varepsilon) \leq F_0(v,\varphi,y_0).$$
\end{thm}

\begin{proof} We may assume that $F_0(v,\varphi,y_0)<+\infty$ otherwise there is nothing to prove. This implies that $v \in \mathcal{M}^{(q)}(\Omega)$, $y \in \spt(\dv v +f)$, and $\varphi =1$ a.e. in $\Omega$. Recall that this implies that  $v \in L^q(\Omega,\R^2)$ and that ${\rm div}\;v$ is a  singular measure with respect to the  Lebesgue measure supported on an $\mathcal{H}^1$-rectifiable set.  Let us also recall that in this case
$$F_0(v,\varphi,y)=
\frac{1}{q}\int_{\Omega} \modvec v\modvec ^q dx + |\dv v|(\Omega)+  \Hh^1_C(v)$$
where
\begin{eqnarray}
\Hh^1_C(v):=\inf\{ \Hh^1(K) \;;\; K \subset \overline{\Omega}\text{ is closed, connected, and } \spt(\dv v +f) \subseteq K\}. \label{infimum}
\end{eqnarray}

Let $K_0$ be the compact and connected set given by Remark \ref{thegoodremark} such that
$$\Hh^1_C(v)=\Hh^1(K_0).$$
We also have that $-\dv v =f$ in the sense of distribution in $\Omega\setminus K_0$. Let  $a_\varepsilon$ and $b_\varepsilon$ be the same parameters than the ones of Lemma \ref{limsup}, let $\varphi_\varepsilon$ be the family of functions given by Lemma \ref{limsup} with $k_\varepsilon=o(\varepsilon)$, and then consider the family of functions
$\varphi_{\varepsilon,\delta_\varepsilon}$ given by Lemma \ref{limsupmodifie} with the choice 
$\delta_\varepsilon:=C(\Omega)(a_\varepsilon+b_\varepsilon).$
We assume that $\varepsilon$ is small enough so that the assumptions  of Lemma \ref{limsupmodifie} are satisified, and  $\varphi_{\varepsilon,\delta_\varepsilon} =1$ on $\partial \Omega$ thanks to \eqref{cool0}. For simplicity we  denote again by $\varphi_{\varepsilon}$ the functions
$\varphi_{\varepsilon,\delta_\varepsilon}$. We know that
$$\limsup_{\varepsilon \to 0} \left( \int_{\Omega}\varepsilon |\nabla \varphi_\varepsilon|^2 + \frac{(1-\varphi_\varepsilon)^2}{4\varepsilon} \right) \leq \Hh^{1}(K_0)=\Hh^1_C(v).$$
We also take $v_\varepsilon=v$ and $y_\varepsilon=y_0$ for all $\varepsilon>0$.

Now looking at each term of $F_{\varepsilon}(v_\varepsilon,\varphi_\varepsilon,y_0)$ containing $v_\varepsilon$, we notice that the only non constant one is 
\begin{eqnarray}
\frac{1}{\sqrt{\varepsilon}}\int_{\Omega}d_{\varphi_\varepsilon}(x,y_\varepsilon)\, d\big(|\dv v +f|\big)(x)  . \label{term2} \label{term3}
\end{eqnarray}
Remember that $|\dv v +f|$ is supported on $K_0$.  Next, we recall that $\varphi_\varepsilon=k_\varepsilon$ on the connected set $(1+\delta_\varepsilon)^{-1}K_0$, and moreover 
$$d(K_0,(1+\delta_\varepsilon)^{-1}K_0)\leq \diam(\Omega) \delta_\varepsilon.$$
Therefore we can find two points $x_1$ and $x_2$ in $(1+\delta_\varepsilon)^{-1}K_0$ such that 
$$\max(d(x, x_1),d(y_0,x_2))\leq \diam(\Omega) \delta_\varepsilon .$$
 Furthermore, since $(1+\delta_\varepsilon)^{-1}K_0$ is path connected, it contains a rectifiable path $\Gamma_x$ connecting $x_1$ to $x_2$ inside $K_0$. By consequence the path 
 $$\Gamma:= [x , x_1] \cup \Gamma_x \cup[x_2, y_0]$$
   is an admissible path in the definition of $d_{\varphi_{\varepsilon}}(x,y_0)$ which yields
   \begin{eqnarray} \label{final1}
0\leq d_{\varphi_{\varepsilon}}(x,y_0)\leq \int_{\Gamma} \varphi_\varepsilon(s) d\Hh^1(s)\leq 2\diam(\Omega)\delta_\varepsilon+k_\varepsilon \mathcal{H}^1(K_0),
\end{eqnarray}
because $\varphi_\varepsilon=k_\varepsilon$ on $\Gamma_x$, is smaller than $1$ everywhere, and because 
$$\mathcal{H}^1(\Gamma_x)\leq \mathcal{H}^1((1+\delta_\varepsilon)^{-1}K_0)\leq \mathcal{H}^1(K_0).$$

 We have just proved that  $d_{\varphi_\varepsilon}(x,y_\varepsilon)\leq C\delta_\varepsilon$ on $K_0$ (we use that $\varepsilon << \delta_\varepsilon$) and it follows that 
\begin{eqnarray}\label{final2}
\frac{1}{\sqrt{\varepsilon}}\int_{\Omega}d_{\varphi_\varepsilon}(x,y_\varepsilon) d|\dv v +f|(x)\leq   \frac{C\delta_\varepsilon}{\sqrt{\varepsilon}}|\dv v +f|(\Omega) \to 0,
\end{eqnarray}
because $\delta_{\varepsilon}=C(a_\varepsilon+b_\varepsilon)=C(\varepsilon^2+2\varepsilon|\ln(\varepsilon)|)<<\sqrt{\varepsilon}$, which implies
$$\limsup_{\varepsilon \to 0} F_{\varepsilon}(v_\varepsilon,\varphi_\varepsilon,y_0) \leq \frac{1}{q}\int_{\Omega} \modvec v\modvec ^q dx + |\dv v|(\Omega)+  \Hh^1_C(v),$$
and finishes the proof.
\end{proof}

\subsection{A note on the existence of minimizers for $\ve>0$}\label{existence issues}

\label{sectionexistence}

The existence of minimizers for the functionals $F_{\varepsilon}(v,\varphi,y)$ when $\varepsilon>0$ is fixed is a very delicate matter, and the same is true for the minimization of $S_\ve$ that we used in Section 4 to approximate the Steiner problem. Indeed, the troubles come from the behavior of the map $\varphi\mapsto d_\varphi$. First, notice that we only restricted our attention to $\varphi\in C^0(\overline{\Omega})$ for the sake of simplicity, in order to get a well-defined $d_\varphi$. Indeed, it is possible to define $d_\varphi$ as a continuous function as soon as $\varphi\in L^p$ for an exponent $p$ larger than the dimension (here, $p>2$, see \cite{CarJimSan}). Since we use functions $\varphi$ which are in $H^1$ in dimension two, they belong to $L^p$ for every $p$, and $d_\varphi$ could be defined in this (weak) sense. The difficult question is which kind of convergence on $\varphi$ provides convergence for $d_\varphi$ (notice that in this setting, as soon as $|\varphi|\leq 1$ any kind of weaker convergence, including pointwise one, implies uniform convergence since all the functions $d_\varphi$ are $1-$Lipschitz). If one wanted upper semi-continuity of the map $\varphi\mapsto d_\varphi(x,x_1)$ (for fixed $x$ and $x_1$), this would be easy, thanks to the concave behavior of $d_\varphi$, and any kind of weak convergence would be enough. Yet, in the case of our interest, we would like lower semi-continuity, which is more delicate. An easy result is the following: if $\varphi_n\to\varphi$ uniformly and a uniform lower bound $\varphi_n\geq c>0$ holds, then $d_{\varphi_n}(x,x_1)\to d_{\varphi}(x,x_1)$. Counterexamples are known if the lower bound is omitted. On the contrary, replacing the uniform convergence with a weak $H^1$ convergence (which would be natural in the minimization of $S_\ve$) is a delicate matter (by the way, the continuity seems to be true and it is not known whether the lower bound is necessary or not), which is the object of an ongoing work with T. Bousch.

However a careful look at our proofs reveals that we could change the space on which the approximating functional is defined as $k_\varepsilon\leq \varphi\leq 1$ instead of $0\leq \varphi\leq 1$, for some $k_\varepsilon\to 0$. The $\Gamma$-convergence result of Section \ref{Section-approx-average} still holds with this little modification, and now, up to add a term of the form    $\ve^{p+1}\int|\nabla\varphi(x)|^pdx$ for $p>2$ to the functional $F_\varepsilon$ (which obviously does not change the $\Gamma$-limit but helps to extract uniformly convergent subsequences), a minimizer do exists for $F_\varepsilon$.

For the Steiner approximation one can follow the same strategy, at the difference that $k_\varepsilon$ must be chosen so that $k_\varepsilon/c_\varepsilon\to 0$, for instance $k_\varepsilon=c_\varepsilon^2$ in order to cancel the term involving $d_\varphi$ at the limit. 

%{\color{blue} We stress that }
%The consequences of this possible lack of semicontinuity and existence for our paper are the following
%\begin{itemize}
%\item In Section 4, Theorem \ref{SteinerTheorem} requires the existence of some minimizers $\varphi_\ve$ to the functionals $S_\ve$. It is empty if minimizers do not exist. A possible way to overcome this problem is to add to $S_\ve$ a term of the form $\ve^{p+1}\int|\nabla\varphi(x)|^pdx$ for $p>2$, and to minimize under the constraint $\varphi\geq\ve$. In this way, minimizing sequences for this new functional are bounded in $W^{1,p}$ and hence equi-continuous. Together with the lower bound on $\varphi$, this gives the existence of a minimizer $\varphi_\ve$. On the other hand, the sequence $\psi_\ve$ which is constructed thanks to Lemma \ref{limsup} already satisfies $\psi_\ve\geq\ve$, and $\lim_{\ve\to 0}\ve^{p+1}\int|\nabla\psi(x)|^pdx=0$ thanks to Remark \ref{estimate ve-p}.
%\item In Section 5, the $\Gamma-$convergence result stays true since semicontinuity and existence of the minimizers for the appproximating functionals are not a necessary condition to state it. In this case, the functionals $F_{\varepsilon}(v,\varphi,y)$ are not semicontinuous since the continuity of $\varphi$ is a constraint which is not closed at all, and not only because the semicontinuity of the term in $d_\varphi$ is not proven. On the other hand, adding a term $\ve^{p+1}\int|\nabla\varphi(x)|^pdx$ produces the same resolutive result as above.
%\end{itemize}

We stress anyway that from the point of view of  the numerical applications this lack of semicontinuity is not crucial, and moreover that no true minimizer is really needed but quasi-minimizers as in Definition \ref{approximateMin} are enough.

%%%%%%%%%%%%%%%%%%%%%%%%%%%%%%%%%%%%%%%%%%%%%%%%%%%%
%%%%%%%%%%%%%%%%%% PARTIE NUMERIQUE %%%%%%%%%%%%%%%%%%%
%%%%%%%%%%%%%%%%%%%%%%%%%%%%%%%%%%%%%%%%%%%%%%%%%%%%

\section{Numerical approximation}

In this section, we apply the relaxation process described in Section \ref{Section-approx-average} to approximate numerically the solution of the compliance problem \eqref{pb20}, in the case $p=2$, and for a constant right-hand side $f\equiv 1$. We decided for the sake of simplicity to stick to a unique problem, and to choose the most ``regular'' one, i.e. the quadratic compliance problem. The approach for the average distance problem for other values of the exponent $p$ would be essentially similar.

We consider a rectangular domain $\Omega\subset \R^2$ and we fix $\Lambda > 0$ and a point $y_0\in\Omega$. For all $\eps > 0$ and for every pair $(v,\varphi)\in C^1(\Omega,\R^2)\times (H^1(\Omega)\cap C^0(\Omega))$, such that
$v \cdot n_\Omega = 0$ on $\partial\Om$ and $
0\leq \varphi \leq 1$,
we define
\begin{align}
G_{\Lambda,\eps}(v,\varphi) = \frac{1}{2} \int_{\Om}|v|^2 + \frac{\Lambda}{4\eps}\int_{\Om}(1-\varphi)^2 + \Lambda \eps\int_{\Om}|\nabla \varphi|^2 \nonumber\\ +\frac{1}{2\sqrt{\eps}}\int_{\Om}\sqrt{(\dv v + 1)^2 + \eps^2}\ \dphi(x,y_0)\ dx. \label{Definition:fonctionnelle-regularisee}
\end{align}
In the definition above,  the last integral is a regularization of the original term 
\[\int_\Om |\dv v+1|\ \dphi(x,y_0)\ dx,\]
allowing for differentiation with respect to $v$.

Our goal is to compute an approximate value of a minimizer $(v_{\Lambda,\eps},\varphi_{\Lambda,\eps})$ of $G_{\Lambda,\eps}$, for a given $\Lambda > 0$ and a small value of $\eps$.
%Let us stress the fact that the functional $G_{\eps,\Lambda}$ is \emph{non convex} with respect to $\varphi$, in general, because the geodesic distance $\dphi({\scriptstyle\bullet},y_0)$ is a concave function of $\varphi$.

\subsection{Discretization of the relaxed problem}

To simplify the notation, we consider the case $\Om=(0,1)^2$ throughout this subsection. The discretization can be adapted straightforwardly to the case of a rectangular domain $\Omega$, discretized by a squared grid of size $h\times h$. We fix $N\in\N^*$, a step $h=\frac{1}{N+1}$ and we define a regular grid, composed of squared cells $(C_{i,j})_{1\leq i,j\leq N+1}$, defined by
\[
C_{i,j}=((i-1)h,ih)\times((j-1)h,jh),\quad \text{for}\ 1\leq i,j\leq N+1.
\]
We denote by $Y_{i,j}$ the center of the cell $C_{i,j}$, defined by $Y_{i,j}=((i-\frac{1}{2})h, (j-\frac{1}{2})h ) $. Following a standard approach in numerical fluid dynamics, the approximations of the scalar fields $\varphi$, $\dv v$ are located on the centers $Y_{i,j}$, whereas the vector fields $\nabla \varphi$, $v$ are discretized on a staggered grid. Namely, the horizontal components of the fields are computed on a $(N+1)\times (N+2)$ grid $X^1$, which is located on the midpoints of the vertical cells interfaces :
\[
X^1_{i,j}=(ih,(j-\frac{1}{2})h)\quad 0\leq i\leq N+1,\ 1\leq j \leq N+1,
\]
and the vertical components are computed on a $(N+2)\times (N+1)$ grid $X^2$, located on the midpoints of the horizontal interfaces :
\[
X^2_{i,j}=((i-\frac{1}{2})h,jh)\quad 1\leq i\leq N+1,\ 0\leq j \leq N+1.
\]

Here is the discretization for the vector field $v$, and its divergence is coherently defined as a scalar field at the center of each square of the grid. We need to ingetrate $d_\phi$ according to a measure involving $\dv v$, so we need to define $d_\phi$ on the same grid. Since the distances $d_\phi$ are computed via a Fast-Marching algorithm which provides values for $d_\phi$ on the same regular grid where $\phi $ is defined, we also define $\phi $ on the very same grid.  In what follows we will denote by $\phi_{i,j}$ the value of a discrete scalar field $\phi$ at point $Y_{i,j}$, and by $V^1_{i,j}$ (resp. $V^2_{i,j}$) the value of the first (resp. second) component of a discrete vector field $V$, at point $X^1_{i,j}$ (resp. $X^2_{i,j}$). Gradients are computed by finite differences as follows

%Gradients are computed by finite differences as follows
%To approximate the derivatives, we apply standard second order finite difference operators. Considering a scalar function $\varphi$, represented on grid $Y$ by a discrete scalar function $\phi$, the derivatives of $\varphi$ along the $x$-axis (resp. the $y$-axis) at points of grid $X_1$ (resp. $X_2$) are approximated using the operators $D_x$ (resp. $D_y$), defined as follows :
\begin{align*}
(D_x\phi)_{i,j} & = \frac{\phi_{i+1,j} - \phi_{i,j}}{h},\quad  1\leq i\leq N,\quad 1\leq j \leq N+1,\\
(D_y\phi)_{i,j} & = \frac{\phi_{i,j+1} - \phi_{i,j}}{h},\quad  1\leq i\leq N+1,\quad 1\leq j \leq N,
\end{align*}

In the same fashion, we approximate at the center of the cells the divergence of a vector field $v$, associated to its discrete representative $V$, using the operator $\Dv$ defined by
\begin{align*}
(\Dv V)_{i,j} & = \frac{V^1_{i,j} - V^1_{i-1,j}}{h} + \frac{V^2_{i,j} - V^2_{i,j-1}}{h},\qquad 1\leq i,j\leq N+1.
\end{align*}
To be consistent with the boundary condition $v \cdot n_\Omega = 0$  satisfied by the continuous vector fields $v$ on $\partial\Om$, we impose the boundary conditions
\begin{align*}
V^1_{0,j} = V^1_{N+1,j} = 0, \quad & 1\leq j \leq N+1,\\
V^2_{i,0} = V^2_{i,N+1} = 0, \quad & 1\leq i\leq N+1.
\end{align*}

\paragraph{The discretized functional.}
We discretize the functional $G_{\Lambda,\eps}$, defined by \eqref{Definition:fonctionnelle-regularisee}, using a first order discretization of the integrals.

We assume that the point $y_0\in\Om$, associated to the geodesic distance $d_{\varphi}({\scriptstyle\bullet},y_0)$, coincides with a certain point $Y_{i^*,j^*}$ of the grid $Y$.
To approximate the geodesic distance, we apply a fast marching algorithm (see \cite{sethian-fastmarching}, \cite{tsitsiklis}) using the vector $\phi$, the discrete representative of $\varphi$ on grid $Y$. We denote by $d_{\phi}$ the corresponding discrete geodesic distance, which is located at points of grid $Y$. Notice that $d_{\phi}$ depends on the indices $(i^*,j^*)$, but we have dropped this dependency to lighten the notation.

Now, we introduce the discrete functional $\Ghlambeps$, obtained by discretizing the functional $G_{\Lambda,\eps}$ using step $h$. To simplify the notation, we introduce the following sets of subscripts $(i,j)$ :
\[
I_1 = [\mspace{-2.5 mu}[1 , N ]\mspace{-2.5 mu}] \times  [\mspace{-2.5 mu}[1, N+1]\mspace{-2.5 mu}],
\quad
I_2 =  [\mspace{-2.5 mu}[ 1 , N+1]\mspace{-2.5 mu}] \times  [\mspace{-2.5 mu}[ 1, N]\mspace{-2.5 mu}]
\quad 
J=  [\mspace{-2.5 mu}[ 1, N+1 ]\mspace{-2.5 mu}]^2.
\]
For every discrete scalar field $\phi$, satisfying $\eta\leq \phi \leq 1$, and every discrete vector field $V$, satisfying the boundary conditions
\begin{align*}
V^1_{0,j} = V^1_{N+1,j} = 0, \quad & 1\leq j \leq N+1,\\
V^2_{i,0} = V^2_{i,N+1} = 0, \quad & 1\leq i\leq N+1,
\end{align*}
 we define $\Ghlambeps(V,\phi)$ by
\begin{align*}
\Ghlambeps(V,\phi) &  = \frac{1}{2} \Big(\sum_{(i,j)\in I_1}h^2 (V^1_{i,j})^2 + \sum_{(i,j)\in I_2} h^2 (V^2_{i,j})^2\Big) \\
& + \frac{\Lambda}{4\eps}\sum_{(i,j)\in J} h^2 (1-\phi_{i,j})^2  + \Lambda\eps \Big(\sum_{(i,j)\in I_1} h^2 [(D_x\phi)_{i,j}]^2 + \sum_{(i,j)\in I_2} h^2 [(D_y\phi)_{i,j}]^2\Big)\\
& + \frac{1}{2\sqrt{\eps}} \sum_{(i,j)\in J} h^2 \sqrt{((\Dv V)_{ij}+1)^2+\eps^2}\ (d_{\phi})_{i,j}.
\end{align*}

Let us stress that the functional $\Ghlambeps$ is strictly convex with respect to $V$, but, in general, non convex with respect to $\phi$. This is a consequence of the concavity of the functions
$
\phi \mapsto (d_{\phi})_{i,j}
$,
(see for instance \cite{bcps}). {\color{black} This concavity is both satisfied in a continuous setting, and in the discrete approximation of the fast marching algorithm.} Thus, the convexity properties of the functional $\Ghlambeps$, with respect to $\phi$, results from the competition between the quadratic term
\[
\frac{\Lambda}{4\eps}\sum_{(i,j)\in J} h^2 (1-\phi_{i,j})^2  + \Lambda\eps \Big(\sum_{(i,j)\in I_1} h^2 [(D_x\phi)_{i,j}]^2 + \sum_{(i,j)\in I_2} h^2 [(D_y\phi)_{i,j}]^2\Big)
\]
and the concave term
\begin{equation}\label{concave-term}
\frac{1}{2\sqrt{\eps}} \sum_{(i,j)\in J} h^2 \sqrt{((\Dv V)_{ij}+1)^2+\eps^2}\ (d_{\phi})_{i,j}.
\end{equation}
{\color{black} As a result, the search for a
global minimizer of $\Ghlambeps$ for arbitrary values of $\Lambda, \varepsilon$ is very delicate. 
%However, in our numerical experiments, we will consider values of $\Lambda$ of order $10$.
However, the convex term (with coefficient $\Lambda/\eps$) can be expected to dominate the concave one (with coefficient $1/(2\sqrt{\eps})$), at least when $\eps$ is large enough. This observation is the key point in the optimization strategy that we present in the next paragraphs, inspired from the works of Oudet in \cite{OudetKelvin,OudetSantam}.}

\subsection{The optimization process}\label{section:optimization-process}

To take into account the specificities of each functional $\Ghlambeps({\scriptstyle\bullet}, \phi)$ and $\Ghlambeps(V,{\scriptstyle\bullet})$, we propose to optimize alternatively in each direction $V, \phi$. For a given $\Lambda > 0$ and a fixed $\eps>0$, we define the following minimization algorithm $(MA)_\eps$ :
\begin{quote}
\hrulefill
\begin{center}
\textsc{Minimisation algorithm $(MA)_\eps$}
\end{center}
\begin{itemize}
\item \textsc{Inputs:} a tolerance $\delta>0$, an initial guess $\phi^0$.
\item \textsc{Output:} a pair $(V_{\eps},\phi_{\eps})$, local minimizer of $\Ghlambeps$ with respect to each direction $({\scriptstyle\bullet},\phi_\eps)$ and $(V_\eps,{\scriptstyle\bullet})$.
\item \textsc{Instructions:}
\begin{enumerate}
\item Define $n=0$ and $V^0\equiv 0$. 
\item Repeat :
\item\label{Step-CG} Find $V^{n+1}$, the global minimizer of $\Ghlambeps({\scriptstyle\bullet},\phi^n)$. 
\item\label{Step-SPG} Find $\phi^{n+1}$, a (local) minimizer of $\Ghlambeps(V_{n+1},{\scriptstyle\bullet})$. 
\item $n:=n+1$.
\item Until $|\Ghlambeps(V^{n},\phi^{n})-\Ghlambeps(V^{n-1},\phi^{n-1})|\leq \delta$.
\item Define $(V_\eps,\phi_\eps):=(V^{n},\phi^{n})$.
\end{enumerate}
\end{itemize}
\hrulefill
\end{quote}
Step \ref{Step-CG} is performed using Fletcher-Reeves nonlinear conjugate gradient algorithm, implemented in the GNU Scientific Library \cite{GSL-Lib}. We refer to \cite{Ref-CG} for a description of the algorithm.
For step \ref{Step-SPG}, to take into account the constraints $\eta\leq \phi \leq 1$, we apply the "Spectral Projected Gradient Method", which is a classical projected gradient method extended to include nonmonotone line search strategy and the spectral steplength. The version that we use is a part of the Open Optimization Library \cite{OOL-Lib}, and implements the algorithm published originally by Birgin \emph{et al.} \cite{Ref-SPG}.

To apply this projected gradient method, we have to address the differentiability of the functional $\Ghlambeps$ with respect to $\phi$, and define its gradient at each step. However, as observed in \cite{bcps}, the functions $\phi \mapsto (d_{\phi})_{i,j}$ fail to be differentiable, in general. This is a consequence of the lack of differentiability of the geodesic distance with respect to the metric. Nevertheless, Benmasour \emph{et al.} prove in \cite{bcps} that the functions $(d_{\phi})_{i,j}$, computed using a fast marching algorithm, are concave (and continuous) functions of $\phi$, and propose an algorithm, called the "subgradient marching", allowing to compute, at each node $(i,j)$, both $(d_{\phi})_{i,j}$ and an element of the superdifferential at a given $\phi$. Using the subgradient marching method, we are able to define a descent direction, obtained by summing the gradient of the quadratic term and an element of the superdifferential of the concave term of $\Ghlambeps$.

{\color{black}Let us emphasize that a global minimizer of $\Ghlambeps$, for a given $\eps$, is hard to find, because of the concavity.} Moreover, the alternate directions optimization strategy does not guarantee that the profiles obtained for $(V,\phi)$ are, in fact, local minimizers; they are local minimizers with respect to perturbations of $V$, or $\phi$, separately.

In order to avoid some local minimizers, we apply a strategy proposed by Oudet in \cite{OudetKelvin} and then applied to other problems. For a given $\Lambda>0$, we consider a target value $\eps_{final}$, and the problem of minimizing $G^h_{\Lambda,\eps_{final}}$. Due to the concave term \eqref{concave-term}, step \ref{Step-SPG} of algorithm $(MA)_{\eps}$, may lead to a local minimizer $\phi^{n+1}$. To avoid this phenomenon, we apply the following iterative optimisation procedure:

\begin{quote}
\hrulefill
\begin{center}
\textsc{Iterative optimisation algorithm}
\end{center}
\begin{itemize}
\item \textsc{Inputs:} 
\begin{itemize}
\item an integer $L\in\N$ and a decreasing family $\lbrace \eps_{\ell} \rbrace_{\ell=0}^L$ of positive numbers, such that $\eps_L=\eps_{final}$;
\item an initial guess $\phi_0$;
\item a tolerance $\delta > 0$.
\end{itemize}
\item \textsc{Output:} a family of local minimizers $(V_{\eps_{\ell}},\phi_{\eps_{\ell}})$, of each functional $G^h_{\Lambda,\eps_{\ell}}$, for $\ell=0,\ldots,L$.
\item \textsc{Instructions:}
\begin{enumerate}
\item Define $(V_{\eps_0},\phi_{\eps_0})$ as the output of algorithm $(MA)_{\eps_0}$, with the tolerance $\delta$ and the initial guess $\phi_0$. 
\item For $\ell=1,\ldots, L$ :
\item Define $(V_{\eps_{\ell}},\phi_{\eps_{\ell}})$ as the output of algorithm $(MA)_{\eps_{\ell}}$, with the initial guess $\phi_{\eps_{\ell-1}}$ and the tolerance $\delta$. 
\end{enumerate}

\end{itemize}
\hrulefill
\end{quote}

\subsection{Numerical results}

We have applied the procedure described in Section \ref{section:optimization-process}, in the case of a rectangle $\Omega = (0,0.5)\times (0,1)$ discretized by a regular grid, composed of squares of size $h=1/100$. The optimisation algorithm was initialized using a quadratic, nonnegative profile for the initial guess $\phi^0$, vanishing only at point $y_0$. This particular choice of $\phi^0$ was motivated by the fact that,
 at convergence of the algorithm, we expect $\phi$ to vanish at point $y_0$, and to take values close to $1$, far from this point.

We present in Figure \ref{Fig:lambda} the optimal profiles for $\phi$, obtained for different values of the parameter $\Lambda$, associated to the penalization of the length of the unknown connected set $K$. In these simulations, we have fixed $y_0=(0.25,0.5)$, that is, the center of the rectangle. The results that are plotted correspond to $\eps = 0.05$. We observe that the sets $\lbrace \varphi = 0\rbrace$ appear as one-dimensional objects, the length of which decreases as $\Lambda$ increases. This feature is consistant with the principle of the penalization. Although we cannot assure that the candidates that we exhibit are, in fact, global minimizers of each functional, this consistency with respect to $\Lambda$ argues in favor of an implicit selection of the minimizers, performed by the algorithm.

We emphasize that, in these examples, the zero level sets of $\phi$ are connected, and contain the point $y_0$. In order to verify that this property still holds for a different position of $y_0$, we have represented in Figure \ref{Fig:decentre} the optimal profile of $\phi$ and the divergence of $V$ obtained with $\Lambda=20$ and $\eps=0.05$, for an off-center grid point $y_0$ (with coordinates $y_0=(0.351485, 0.59901)$). On the plot of $\dv V$, this point can be identified as the most singular point for the divergence. As in the former examples, the zero level set of $\phi$ appears as a connected set containing $y_0$. As a result, we may infer that our numerical method is able to force the zero level set of $\phi$ to be connected to a given grid point.

% dessins pour differents lambda

\begin{center}
\begin{figure}
\begin{center}
	\begin{minipage}[t]{220pt}
		\includegraphics[scale=0.40]{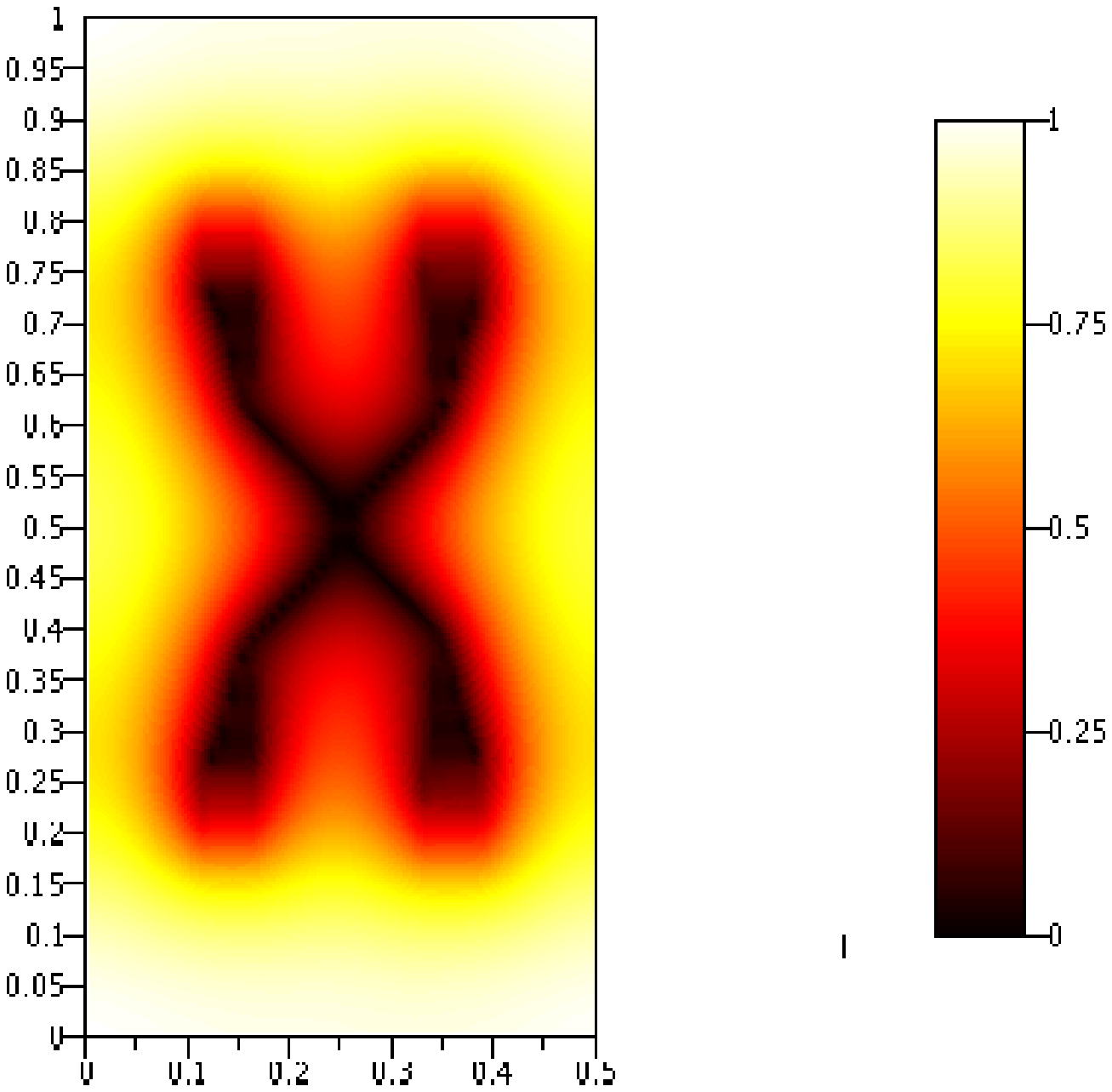}
		\caption*{$\Lambda=10$}
	\end{minipage}
	\begin{minipage}[t]{220pt}
		\includegraphics[scale=0.40]{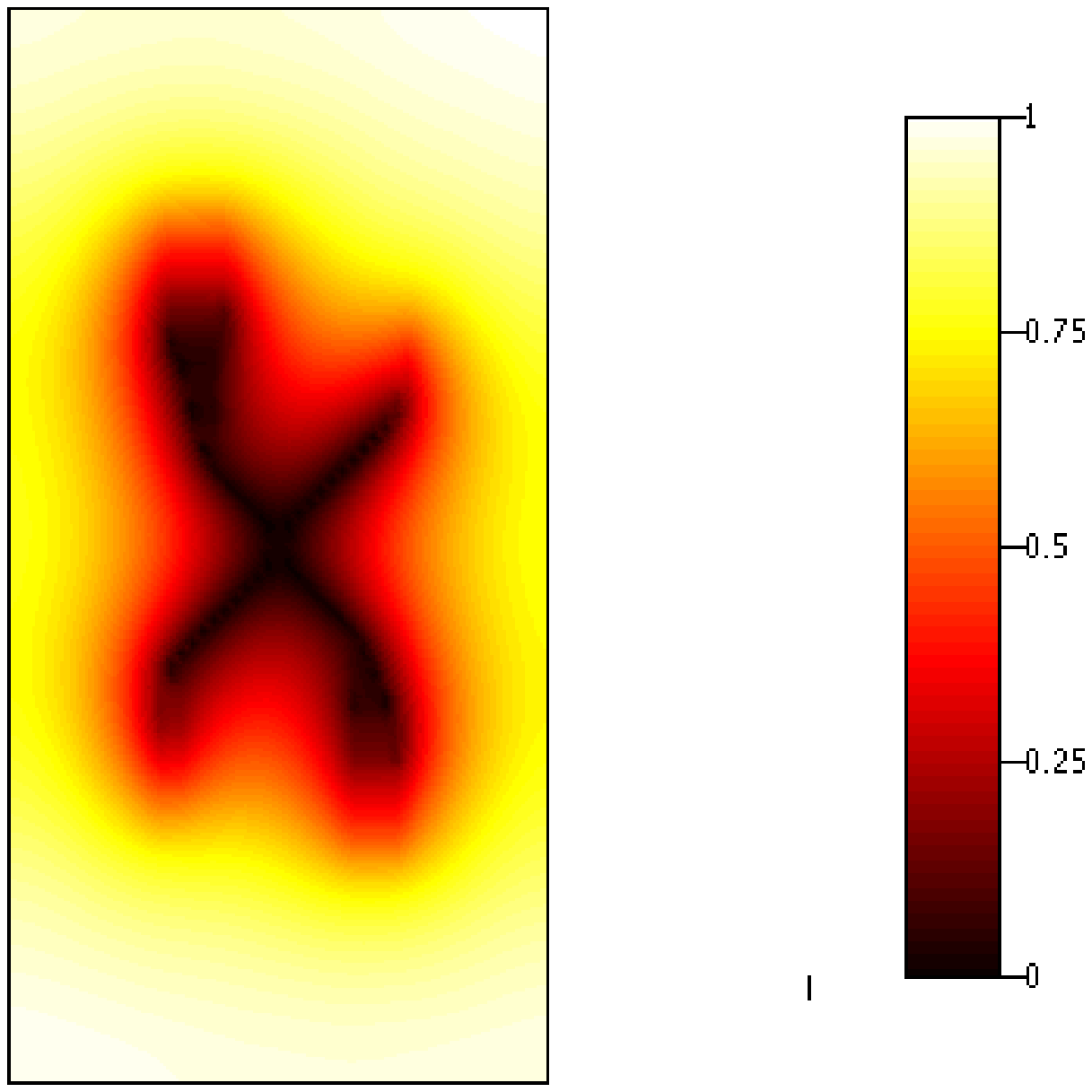}
		\caption*{$\Lambda=15$}
	\end{minipage}
\end{center}
\begin{center}
	\begin{minipage}[t]{220pt}
		\includegraphics[scale=0.40]{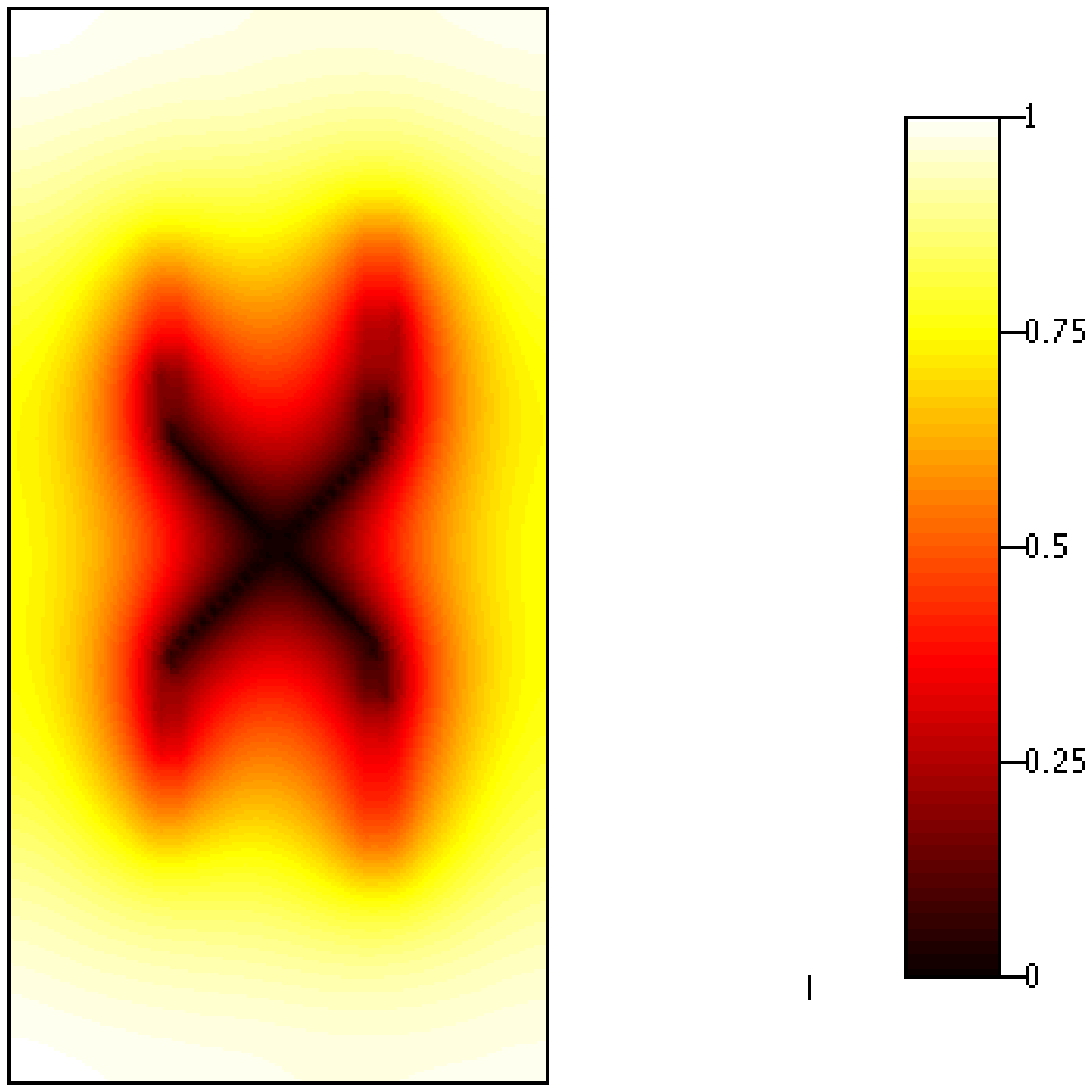}
		\caption*{$\Lambda=20$}
	\end{minipage}
	\begin{minipage}[t]{220pt}
		\includegraphics[scale=0.40]{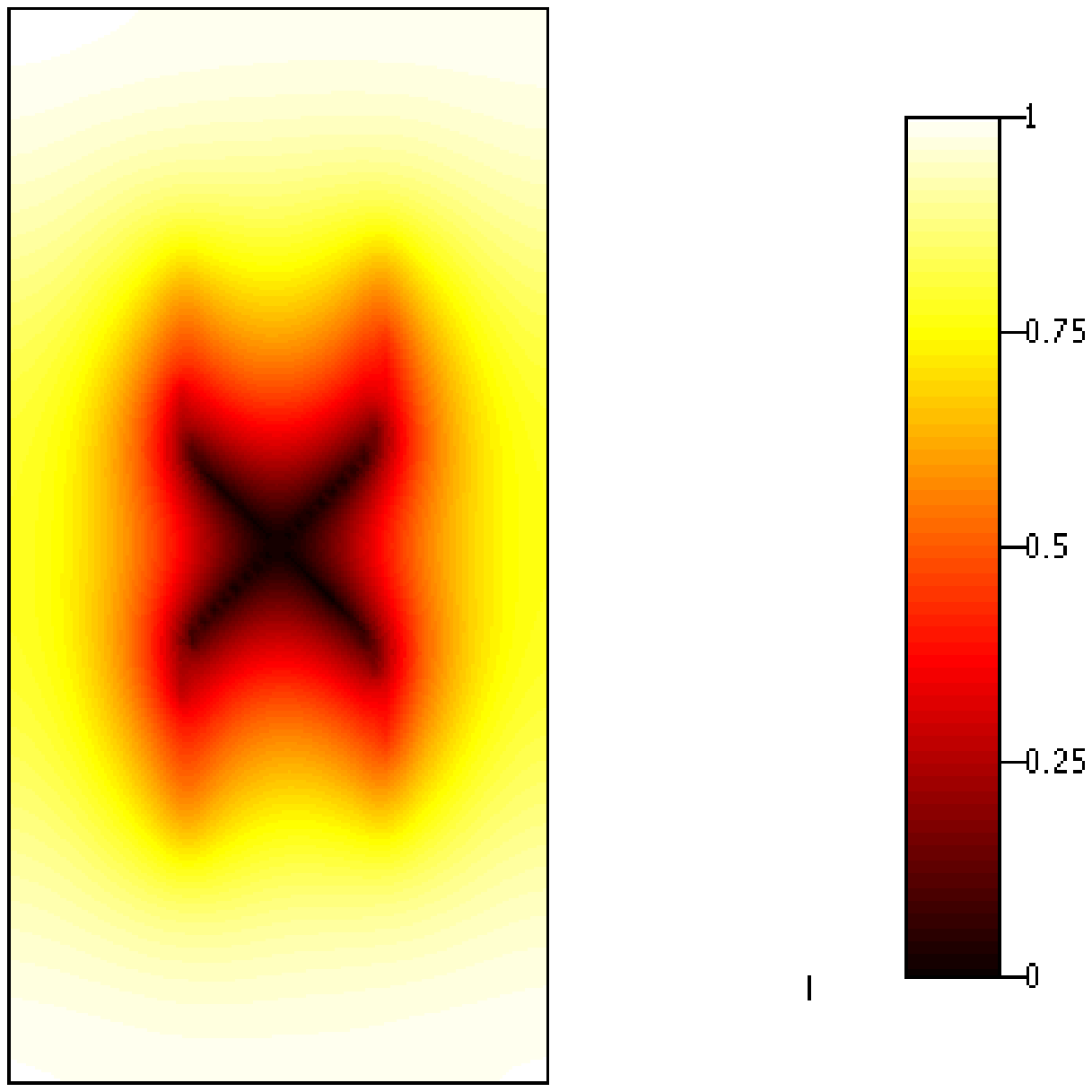}
		\caption*{$\Lambda=30$}
	\end{minipage}
\end{center}
\caption{Optimal profiles for $\phi$, associated to different values of $\Lambda$, with $\eps=0.05$. The point $y_0$ is the center of the rectangle.}\label{Fig:lambda}
\end{figure}
\end{center}

% dessin pour y_0 decentre

\begin{center}
\begin{figure}
\begin{center}
	\begin{minipage}[t]{220pt}
		\includegraphics[scale=0.40]{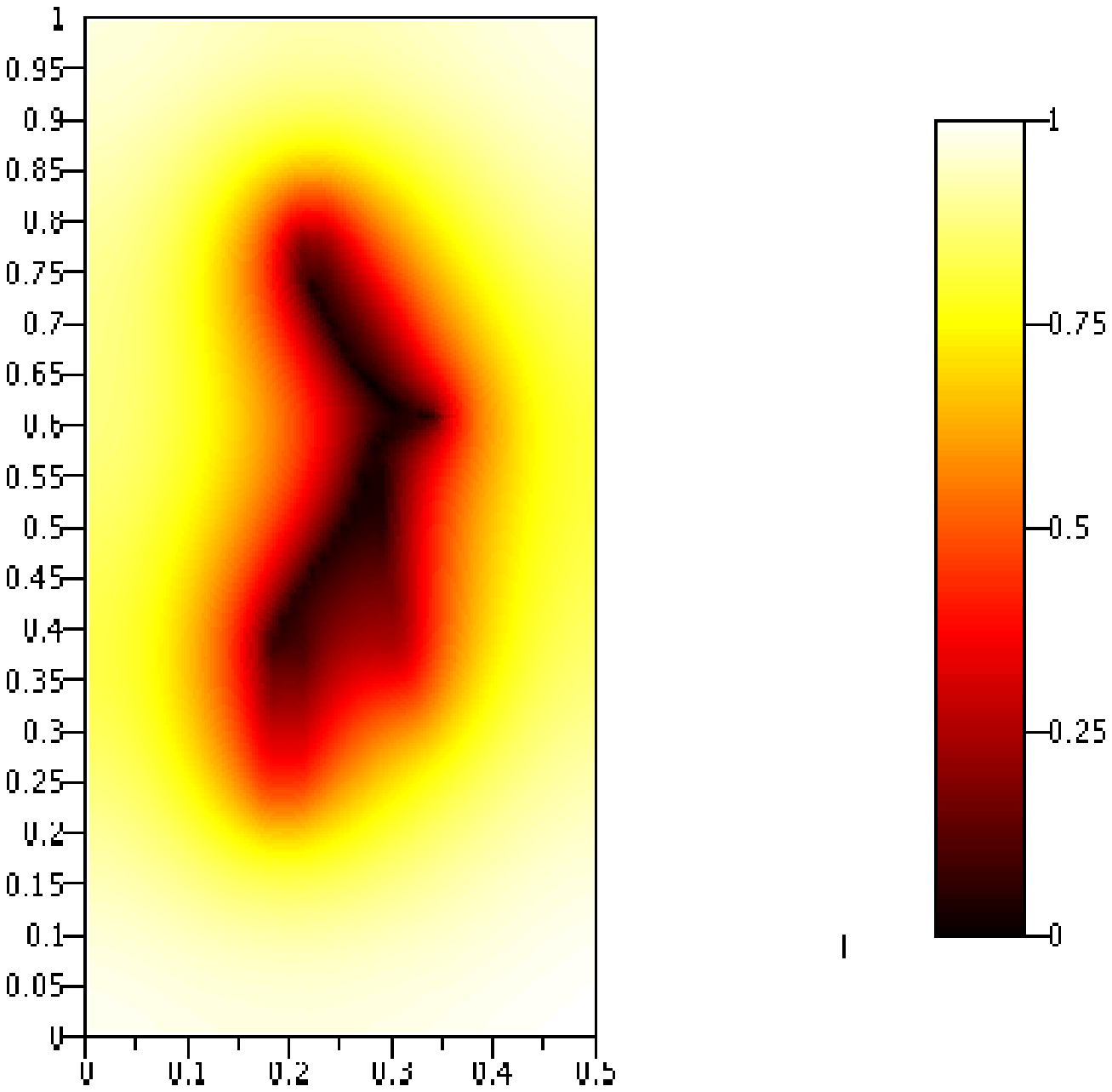}
		%\caption*{$\Lambda=10$}
	\end{minipage}
	\begin{minipage}[t]{220pt}
		\includegraphics[scale=0.40]{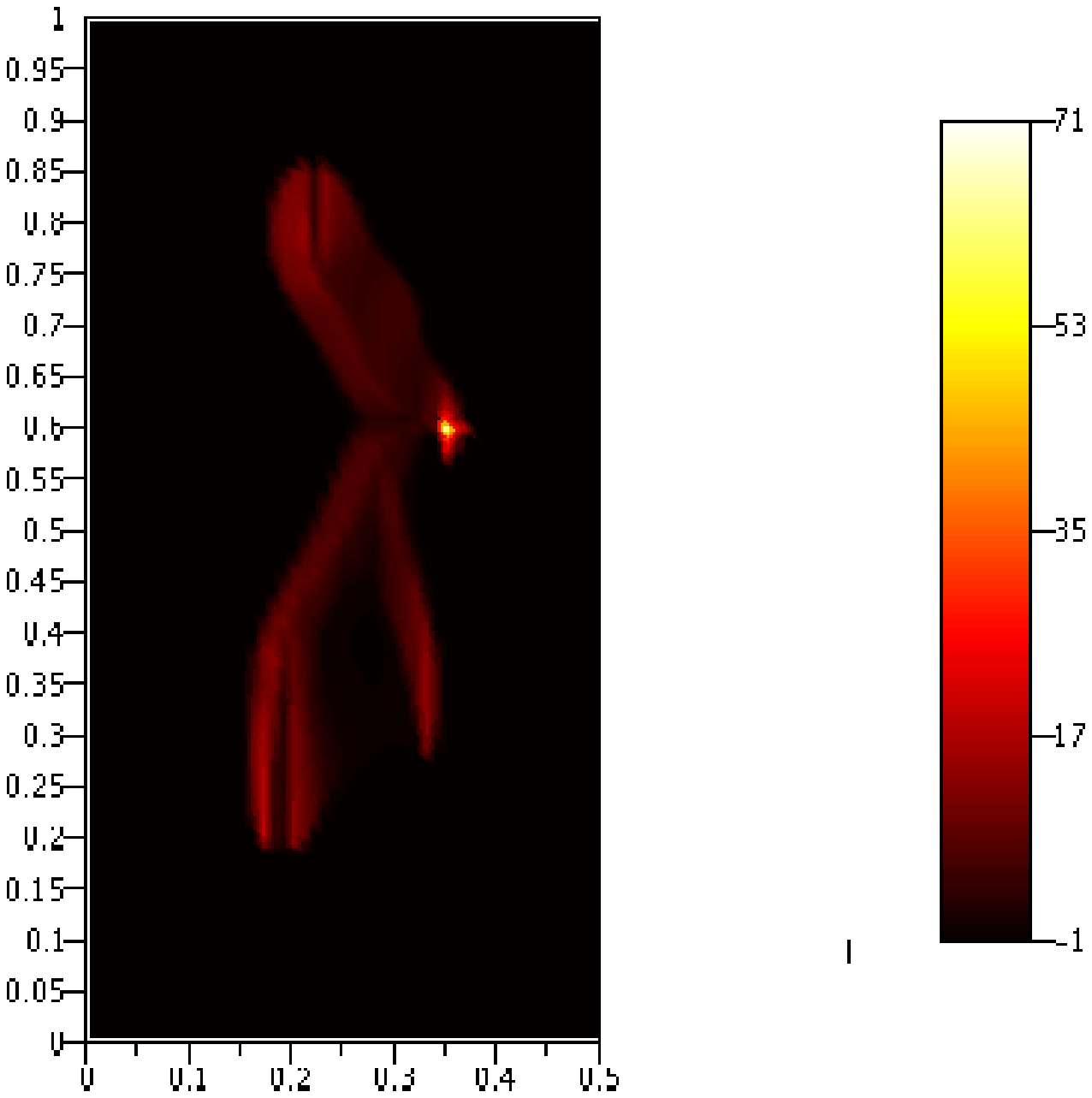}
		%\caption*{$\Lambda=15$}
	\end{minipage}
\end{center}
\caption{Optimal profiles for $\phi$ (left) and $\dv V$ (right), with $\eps=0.05$, in the case of an off-center point $y_0$, with coordinates  $y_0=(0.351485, 0.59901)$. This point coincides with the most singular point of the divergence of $V$. }\label{Fig:decentre}
\end{figure}
\end{center}

\newpage
\bibliographystyle{plain}

%\bibliography{biblio}

 \end{document}